\documentclass[10pt,twocolumn,letterpaper]{article}
\usepackage[pagenumbers]{cvpr} % To force page numbers, e.g. for an arXiv version

\usepackage[dvipsnames]{xcolor}

\definecolor{cvprblue}{rgb}{0.21,0.49,0.74}
\usepackage[pagebackref,breaklinks,colorlinks,citecolor=cvprblue]{hyperref}
\usepackage[accsupp]{axessibility}
\usepackage{graphicx,wrapfig,lipsum}
\usepackage{amsmath}
\usepackage{amssymb}
\usepackage{booktabs}
\usepackage{algorithm}
\usepackage{algorithmicx}
\usepackage{CJK}
\usepackage{algpseudocode}
\usepackage{ntheorem}

\newtheorem*{theorem}{Theorem}
\newtheorem{lemma}{Lemma}
\newtheorem*{proof}{Proof}

\newtheorem{proposition}{Proposition}
\newtheorem*{example}{Example}
\usepackage{caption}
\usepackage{algpseudocode}

\def\paperID{14149} % *** Enter the Paper ID here
\def\confName{CVPR}
\def\confYear{2024}

\title{Adaptive Softassign via Hadamard-Equipped Sinkhorn} 

\author{Binrui Shen\\
Department of Applied Mathematics, School of Mathematics and Physics, 
\\
Xi’an Jiaotong-Liverpool University, Suzhou 215123, P.R. China\\
Department of Mathematical Sciences, School of Physical Sciences, \\ University of Liverpool, Liverpool, United Kingdom\\
{\tt\small binrui.shen19@student.xjtlu.edu.cn}
\and
Qiang Niu\\
Department of Applied Mathematics, School of Mathematics and Physics, \\ Xi’an Jiaotong-Liverpool University, Suzhou 215123, P.R. China\\
{\tt\small qiang.niu@xjtlu.edu.cn}\and
Shengxin Zhu$^\nmid$\\
Research Centers for Mathematics, Advanced Institute of Natural Science,\\ Beijing Normal University, Zhuhai 519087, P.R.China\\
Guangdong Provincial Key Laboratory of Interdisciplinary Research and Application for Data Science, \\
BNU-HKBU United International College, Zhuhai 519087, P.R. China\\
{\tt\small Shengxin.Zhu@bnu.edu.cn}}

\begin{document}
\maketitle
\begin{abstract}

% perform well for small graph matching tasks,
Softassign is a pivotal method in graph matching and other learning tasks. Many softassign-based algorithms exhibit performance sensitivity to a parameter in the softassign. However, tuning the parameter is challenging and almost done empirically. This paper proposes an adaptive softassign method for graph matching by analyzing the relationship between the objective score and the parameter. This method can automatically tune the parameter based on a given error bound to guarantee accuracy. The Hadamard-Equipped Sinkhorn formulas introduced in this study significantly enhance the efficiency and stability of the adaptive softassign. Moreover, these formulas can also be used in optimal transport problems. The resulting adaptive softassign graph matching algorithm enjoys significantly higher accuracy than previous state-of-the-art large graph matching algorithms while maintaining comparable efficiency. 
\end{abstract}
%Compared with the previous state-of-the-art algorithms, such as the scalable Gromov-Wasserstein Learning (S-GWL), the resulting algorithm enjoys both a higher accuracy and a significant improvement in efficiency for large graph matching problems. In particular, on the protein network matching benchmark problems (1004 nodes), our algorithm can improve the accuracy from $56.3\%$ by the S-GWL to  $75.1\%$, and at the same time, it can achieve 3X+ speedup in efficiency.     
\section{Introduction}
\label{sec:intro}
Graph matching aims to find a correspondence between two graphs. As a fundamental problem in computer vision and pattern recognition, it is widely used in shape matching \cite{2010Learning,2011Scale}, detection of similar pictures\cite{shen2020fabricated}, medical imaging \cite{2012med},
graph similarity computation \cite{lan2022aednet,lan2022more} and face authentication \cite{wiskott1999face, kotropoulos2000frontal}. It can even be used in activity analysis \cite{chen2012efficient} and recently in bioinformatics \cite{xu2019scalable}. 

The general graph matching is an NP-hard problem, because of its combinatorial nature \cite{1963The}. Therefore, recent works on graph matching mainly focus on continuous relaxation to obtain a sub-optimal solution with an acceptable cost by constructing approximate optimization methods. Popular approaches include, but are not limited to, spectral-based methods \cite{umeyama1988eigendecomposition, leordeanu2005spectral, luo2003spectral, robles2007riemannian},
continuous path optimization \cite{zaslavskiy2008path,maron2018probably,wang2017graph},  random walk \cite{  2010Reweighted} and probabilistic modeling \cite{egozi2012probabilistic} and optimal transport methods \cite{xu2019gromov, xu2019scalable}.

% These algorithms can produce a continuous solution in a relaxed domain. The continuous solution is then converted back to the discrete domain to obtain the final solution. 

Among recently proposed graph matching algorithms, projected gradient-based algorithms \cite{gold1996graduated, cour2006balanced, leordeanu2009integer,lu2016fast, shen2022dyna} have drawn a lot of attention due to their competitive performances in large graph matching problems. These algorithms iteratively update the solution by projecting gradient matrices into a feasible region, typically addressing a \textit{linear assignment problem}. The performance of these algorithms mainly depends on the underlying projection methods. Among projections, the discrete projection may lead the matching algorithm \cite{gold1996graduated} to converge to a circular sequence  \cite{tian2012convergence}; the doubly stochastic projection used in \cite{lu2016fast} suffers from poor convergence when the numerical values of the input matrix are large \cite{rontsis2020optimal}. Softassign is a more flexible method that allows for a trade-off between efficiency and accuracy. It is proposed to solve linear assignment problems in \cite{kosowsky1994invisible} and is later used as an approximate projection method in graph matching \cite{gold1996graduated}. It consists of an exponential operator and the Sinkhorn method \cite{sinkhorn1964relationship} to achieve inflation and bistochastic normalization, respectively. The inflation step can effectively attenuate unreliable correspondences while simultaneously amplifying reliable ones \cite{2010Reweighted}.

The performance of the softassign-based graph matching algorithms depends largely on the inflation parameter in the inflation step \cite{shen2022dyna}. Previous algorithms tune this parameter empirically \cite{gold1996graduated,2010Reweighted,shen2022dyna,zheng2020fast}. To address such an inconvenience and improve accuracy, this paper proposes an adaptive softassign method. The main contributions of this paper are summarized as follows:
\begin{itemize}
    \item \textbf{Adaptive softassign.} We propose an adaptive softassign method for large graph matching problems. It is designed to automatically tune the parameter according to a given error bound, which can be interpreted as the distance from optimal performance. 
    \item \textbf{Sinkhorn operation rules}. Several introduced convenient operation rules for the Sinkhorn method significantly accelerate the adaptive softassign and increase the stability in Sinkhorn iterations. Furthermore, all theoretical results regarding softassign can be readily extended to the optimal transport problems \cite{cuturi2013sinkhorn}.%, which may shed light on the optimal transport problems.
    \item \textbf{Graph matching algorithm.} 
By combining the adaptive softassign method with a project fixed-point approach, we propose a novel adaptive softassign matching algorithm (ASM). It enjoys significantly higher accuracy than previous state-of-the-art large matching algorithms. See Figure \ref{fig:performance} for comparison.
\end{itemize}
\begin{figure}
    \centering
    \includegraphics[width=0.99\linewidth]{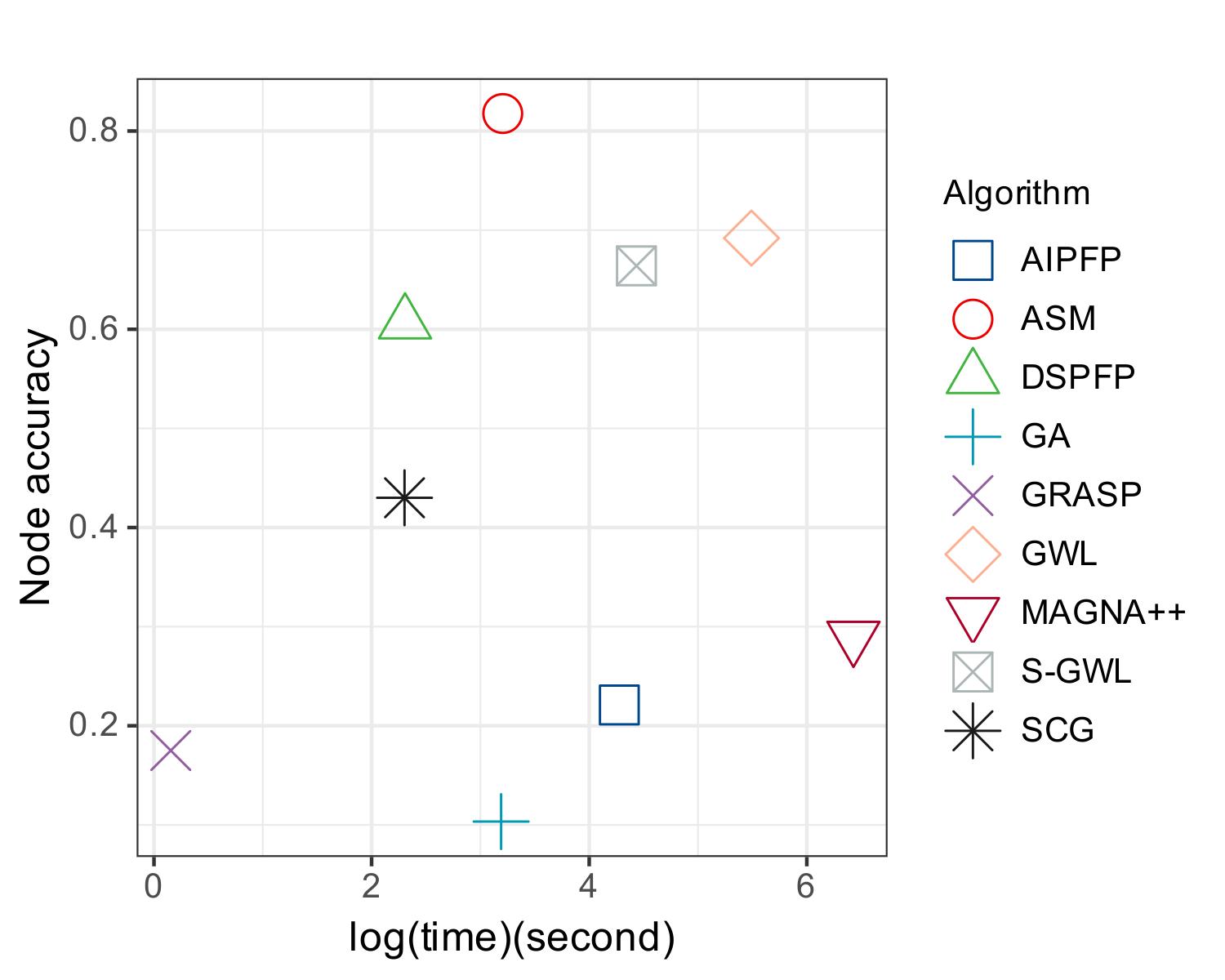}
    \caption{Mean matching accuracy and running time of different algorithms on protein network matching.}%(25\% noise level)
    \label{fig:performance}
\end{figure}
The structure of this paper is as follows. Section 2 introduces the graph matching problem, the projected fixed-point method, and softassign. Section 3 showcases adaptive softassign, Sinkhorn formulas, and the potential impact of Sinkhorn formulas on the optimal transport problem. Section 4 discusses algorithmic details of adaptive softassign matching. An empirical study is conducted in Section 5 before concluding. Theoretical proofs are shown in the Appendix.

\newcommand{\bfX}{\mathbf{X}}
\newcommand{\calR}{\mathbb{R}}

%-------------------------------------------------------------------------

\section{Preliminaries} 
Table \ref{tab:notation} summarizes the main symbols and notations used in this paper. We use lowercase letters for scalars (e.g., $\beta$), bold lowercase letters for vectors (e.g., $\mathbf{1}$), and uppercase letters for matrices (e.g., $A$).
% Let $\mathbf{1}$ and $ \mathbf{0}$ be the vectors with all-ones and all-zeroes respectively and denote $ {X} = ( {X}_{ij})$ as a matrix. 
% $\exp( {X})=(\exp( X_{ij}))$ is the matrix with entries $\exp ( {X}_{ij})$.  
% $\| {X}\|_{Fro} = \sqrt{\sum  {X}_{ij}^2}$ is the Frobenius norm.
% For a square matrix $ {X} \in \mathbb{R}^{n\times n}$, we denote its trace by $\operatorname{tr}(X) = \sum X_{ii}$. 
% For two matrices of the same dimension, we denote the Frobenius inner product of $ {X}$, $ {Y}$ by $\langle  {X},  {Y} \rangle = \sum_{ij}  {X}_{ij} {Y}_{ij}$ and the Hadamard product of $ {X} \circ  {Y} = ( {X}_{ij}  {Y}_{ij})$.
% For a constant $c$ and a matrix $ {X}$, we denote the Hadamard power $ {X}^{\circ c} =( {X}_{ij}^{c})$. 
% For a vector $ \mathbf{x} \in \mathbb{R}_n$, we write ${D}_{( {\mathbf{x} })}=\operatorname{diag}\{\mathbf{x}_1, \mathbf{x}_2,\ldots, \mathbf{x}_n\}$ is a diagonal matrix. 
% The set of permutation matrices is denoted by $\Pi_{n \times n} :=\{ {M}: {M} \mathbf{1}= \mathbf{1}, {M}^{T} \mathbf{1}= \mathbf{1},  {M}\in\{0,1\}^{n\times n}\}$ and the set of doubly stochastic matrices is denoted by $\Sigma_{n \times n} :=\{ {M}:  {M} \mathbf{1}= \mathbf{1}, {M}^{T} \mathbf{1}= \mathbf{1},  {M} \geq 0\}$ is the convex hull of $\Pi_{n \times n}$ \cite{cuturi2013sinkhorn}. 
% \usepackage{graphicx}
\begin{table}[]
\centering
\caption{Symbols and Notations.}
\label{tab:notation}
\resizebox{\columnwidth}{!}{%
\begin{tabular}{cc}
\hline
Symbol                       & Definition                         \\ \hline
$\mathbb{G},\tilde{\mathbb{G}}$                 & matching graphs                   \\
$A,\tilde{A}$                          & edge attribute matrices of  $\mathbb{G}$ and $\tilde{\mathbb{G}}$                \\
$F,\tilde{F}$                          & node attribute matrices of  $\mathbb{G}$ and $\tilde{\mathbb{G}}$          \\
$n,\tilde{n}$                          & number of nodes of  $\mathbb{G}$ and $\tilde{\mathbb{G}}$                  \\
$M$                          & matching matrix                   \\
$\Pi_{n \times n}$           & set of $n \times n$ permutation matrices       \\
$\Sigma_{n \times n}$        & set of $n \times n$ doubly stochastic matrices \\\hline
$\mathbf{1}$,$ \mathbf{0}$   & a column vector of all 1s,0s      \\ 
${D}_{( {\mathbf{x} })}$ & diagonal matrix of a vector $\mathbf{x}$                      \\
$\operatorname{tr}(\cdot)$                  & trace                             \\
$\langle \cdot,\cdot\rangle$ & inner product                     \\
$\|\cdot \|_{Fro}$                    &Frobenius norm                 \\
$\exp$                       & element-wise exponential                \\

$\oslash$                    & element-wise division            \\
$\circ$                      & Hadamard product                  \\
$\cdot^{\circ}$              & Hadamard power                    \\ \hline
$\beta$                      & the parameter in softassign       \\
$\mathcal{P}_{sk}(\cdot)$    & Sinkhorn method                   \\
${S}^{\beta}_{ {X}}$     & a matrix from applying softassign with $\beta$ on a matrix $X$ \\ \hline
\end{tabular}%
}
\end{table}
% Let $\mathcal{P}: \mathbb{R}^{n\times n} \mapsto \Sigma_{n\times n}$
% be the project operator onto the set of doubly stochastic matrices, then  $\mathcal{P}( {X})$ is the closest point(s) in $ \Sigma_{n\times n}$ to $ {X} \in \mathbb{R}^{n \times n}$. Similarly, $\mathcal{P}_{dis}: \mathbb{R}^{n\times n} \mapsto \Pi_{n\times n}$ is a projection onto the set of permutation matrices.

% $\|\|_{Fro}$
%and the element-wise division $ {X} \oslash  {Y} = ( {X}_{ij} /  {Y}_{ij})$
% $⊘$
\subsection{Background}
A graph $\mathbb{G}=\{\mathbb{V},\mathbb{E}, {A}, {F}\}$ consists of a node set $\mathbb{V}$ and an edge set $\mathbb{E}$. Further, we can use a symmetric matrix $A$ to denote the attributes of edges and $F$  to store the attributes of each node. 

% The cardinality of the vertice $| \mathbb{V}|$ is simply referred to as the graph size here and thereafter. 

\textbf{Matching matrix} The matching correspondence of two graphs with the same number of nodes is usually represented by a \textit{permutation matrix}  $ {M}=( {M}_{ij})$
\begin{equation}
 {M}_{i j}=\left\{\begin{array}{ll}
1 & \text { if } \mathbb{V}_i  \text { corresponds to  } \tilde{\mathbb{V}}_j, \\
0 & \text { otherwise, }
\end{array}\right.
\end{equation}
where $ {M}\in \Pi_{n \times n}.$

\textbf{The graph matching problem} can be formulated as a quadratic assignment problem minimizing the dissimilarity of two graphs \cite{zaslavskiy2008path}:
\begin{equation}
\min _{ {M} \in \Pi_{n \times n}} \frac{1}{2}\left\| {A}- {M}\widetilde{ {A}} {M}^{T}\right\|_{Fro}^{2}+\lambda\left\| {F}- {M} \widetilde{ {F}}\right\|_{Fro}^{2},
\label{eq.object 1}
\end{equation}
where the left term presents the dissimilarity between edges and the right term presents the dissimilarity between nodes. Since $\| {X}\|_{Fro}^2={\operatorname{tr}\left( {X} {X}^T\right)}$, problem \eqref{eq.object 1} can be rewritten as 
\begin{equation}
\max _{{ {M} \in \Pi_{n \times n}}}\quad \frac{1}{2} \operatorname{tr}\left({ {M}}^{T} { {A}} { {M}} {\widetilde{ {A}}}\right)+\lambda \operatorname{tr}\left({ {M}}^{T}{ {K}}\right),
\label{eq.Object_KB}
\end{equation}
where $ {K} =  {F}\tilde{ {F}}^T$, see \cite{lu2016fast} for more details. 
%This is a quadratic problem.

\textbf{Relaxation method} Due to the discrete constraints, \eqref{eq.Object_KB}  is an NP-hard problem \cite{1963The}. A common trick for solving such discrete problems is relaxation: one first finds a solution $ X$  on a continuous domain $\Sigma_{n\times n}$,
\begin{equation}
 N^* =\arg\max _{ {N} \in \Sigma_{n \times n}}\  \frac{1}{2} \operatorname{tr}\left({ {N}}^{T} { {A}} { {N}} {\widetilde{ {A}}}\right)+\lambda \operatorname{tr}\left({ {N}}^{T}{ {K}}\right),
\label{eq.Object_relaxedKB}
\end{equation}
and $N^*$ is transformed back to the original discrete domain $\Pi_{n\times n}$ by solving a \textit{linear assignment problem} of the following form
\begin{equation}
     M^* = \arg \min_{ {M}\in \Pi_{n \times n}} \| M -  {N^*}\|_{Fro}.
    \label{eq. linear assignment}
\end{equation}
% where $ {X} \in \mathbb{R}^{n\times n}$ corresponds to the continuous solution in \eqref{eq.Object_relaxedKB}. 
The matrix $M^*$ is the final solution for graph matching, which is commonly obtained by the Hungarian method \cite{kuhn1955hungarian} or the greedy method  (efficient but not exact) \cite{luo2003spectral}.
%and $ {Q} = \mathcal{P}_{dis}( X)$
\subsection{Adaptive projected fixed-point method}

%The model \eqref{eq.Object_relaxedKB} takes advantages of efficient matrix matrix operation, and 
%To solve \eqref{eq.Object_relaxedKB} efficiently, \cite{lu2016fast} introduced the projected fixed-point method. 

Consider the objective function 
\begin{equation}
    \mathcal{Z} ( {M}) = \frac{1}{2} \operatorname{tr}\left({ {M}}^{T} { {A}} { {M}} {\widetilde{ {A}}}\right)+\lambda \operatorname{tr}\left({ {M}}^{T}{ {K}}\right).    \label{eq.zm}
\end{equation}
With the help of matrix differential \cite{Harville2008Matrix}, one can obtain the `gradient' of the objective function with respect to $ {M}$
\begin{equation}
    \nabla \mathcal{Z}( {M}) = \frac{\partial \mathcal{Z}( {M})}{\partial  {M}} =    {A} {M}\widetilde{ {A}}+\lambda  {K}.
\end{equation}
The adaptive projected fixed-point method is
\begin{equation}
\begin{aligned}
         {M}^{(k)} = (1&- \alpha)  {M}^{(k-1)}  + \alpha {D}^{(k)},
         \\
         {D}^{(k)} &= \mathcal{P}(\nabla \mathcal{Z}( {M}^{(k-1)})),
\end{aligned}
\label{eq.pfp}
\end{equation}
where $\alpha \in [0,1]$ is a step size parameter and $\mathcal{P}(\cdot)$ is a projection operator used to project the gradient matrix to a feasible region.

An adaptive strategy of the step size parameter proposed in \cite{shen2022dyna} can guarantee the convergence of \eqref{eq.pfp} with any projection type. The optimal step size parameter $\alpha^*$ is determined according to a 'linear search' type technique:
\begin{equation}
  (\alpha^*)^{(k)} = \arg \max_{\alpha}  \  \mathcal{Z}((1-\alpha) {M}^{(k-1)}+\alpha {D}^{(k)}).   
  \label{eq: adaptive alpha}
\end{equation}

According to underlying constraints, projections include the discrete projection used in the integer projected fixed-point method \cite{leordeanu2009integer} and the doubly stochastic projection used in the doubly stochastic projected fixed-point method (DSPFP) \cite{lu2016fast}. The discrete projection solves problem \eqref{eq. linear assignment} and the doubly stochastic projection aims to find the closet doubly stochastic matrix to a given matrix $X$ by solving
\begin{equation}
    Y^* =\arg \min_{Y \in \Sigma_{n \times n}} \|X - Y\|_{Fro},
    \label{eq.dsprojection}
\end{equation}
which equals
\begin{equation}
    Y^* =\arg \max_{Y \in \Sigma_{n \times n}} \ \langle X,Y\rangle.
    \label{eq.projection}
\end{equation}
 
In essence, the projected fixed point method solves a series of linear assignment problems to approximate the solution of problem \eqref{eq.Object_relaxedKB}. The performance of algorithms depends on the quality of solutions to linear problems (projections).

\subsection{Softassign}
Among projection methods, the discrete projection suffers from information loss when the linear assignment problem with discrete constraints has multiple solutions; the doubly stochastic projection suffers from poor convergence when the numerical value of the input matrix is large \cite{rontsis2020optimal}. 
To address these issues,  an entropic regularization term is added to smooth the problem \eqref{eq.projection}:
\begin{equation}
\begin{aligned}
\label{eq: entropic assignment}
         {S}^{\beta}_{ {X}} &= \arg \max_ {S \in \Sigma_{n \times n}} \langle  {S},  {X} \rangle + \frac{1}{\beta}\mathcal{H}( {S}),\\
        &\mathcal{H}( {S}) = -\sum  {S}_{ij} \ln { {S}_{ij}},
\end{aligned}
\end{equation}
where $ {X} = \nabla \mathcal{Z}( {M}^{(k)})$ in the projected fixed-point method for graph matching. As the inflation parameter $\beta$ increases, ${S}^{\beta}_{ {X}}$ approaches the optimal solution of the linear assignment problem \eqref{eq.projection}.

% to infinity, $ {S}^{\beta}_{ {X}}$ approaches a doubly stochastic matrix which is a convex combination of all possible solutions instead of just one solution
Softassign solves \eqref{eq: entropic assignment} to approximate the solution of \eqref{eq.projection} \cite{kosowsky1994invisible}. It has been widely used in graph matching \cite{shen2020fabricated, 2010Reweighted, gold1996graduated}; its general form has been widely used in optimal transport \cite{cuturi2013sinkhorn}. The solution $ {S}^{\beta}_X$ is unique of form  \cite{cuturi2013sinkhorn}
\begin{equation}
( {S}^{\beta}_X)_{ij} =  \mathbf{r}_i {J}_{ij}  \mathbf{c}_j, \ \  \  {J} = \exp(\beta  {X}), \ \ \   \mathbf{r},  \mathbf{c} \in  \mathbb{R}^n_+ .
\label{eq:softassign_exp}
\end{equation}
%where two balancing vectors
 In matrix form, the solution reads as
\begin{equation}
    {S}^{\beta}_{ {X}} =  {D}_{( \mathbf{r})}  {J}  {D}_{( \mathbf{c})}. 
   \label{eq:softassign_matrix}
\end{equation}
To improve numerical stability, we perform a preprocessing on $ {J}$ according to \cite{shen2022dyna}:
% \begin{equation}
%     \hat{ {J}}  = \exp(\beta ( {X}-  {1}  {1}^{T}\max( {X}))) = \frac{ {J}}{\max( {J})},
%     \label{eq:softassign_infla}
% \end{equation}
\begin{equation}
    \hat{ {J}}  = \exp(\beta ( {X}/\max( {X}))),
    \label{eq:softassign_infla}
\end{equation}
where $\max( {X})$ is the maximum element of $ {X}$.
The two balancing vectors $ \mathbf{r}$ and  $ \mathbf{c}$ can be computed by Sinkhorn iterations 
\begin{equation}
 \mathbf{r}^{(\ell+1)} \stackrel{\text {  }}{=}  {\mathbf{1}} \oslash ({\hat{ {J}}  \mathbf{c}^{(\ell)}}) \quad \text { and } \quad  \mathbf{c}^{(\ell+1)} \stackrel{\text {  }}{=}  {\mathbf{1}} \oslash ({\hat{ {J}}^{\mathrm{T}}  \mathbf{r}^{(\ell+1)}}).
    \label{eq:softassign_sinkhorn}
\end{equation}
To summarize, the softassign algorithm consists of two components: inflation by matrix element-wise exponential in \eqref{eq:softassign_infla} and doubly stochastic normalization by the Sinkhorn method in \eqref{eq:softassign_sinkhorn}. The inflation step magnifies large values and diminishes small ones to reduce the effect of unreliable correspondences. Figure \ref{fig_example} illustrates the effect of the $\beta$.

\begin{center}
\begin{minipage}{0.3\textwidth}\raggedright
$$
 {X}=\left(\begin{array}{ccc}
1 & 0.9 & 0.9 \\
0.9 & 1 & 0.5 \\
0.6 & 0.25 & 1
\end{array}\right)
$$
\end{minipage}
% \captionof{figure}{Probability distribution with loading level of a): 0.8 ; b) 0.9}
\begin{minipage}{0.45\textwidth}\raggedleft
\includegraphics[width=\linewidth]{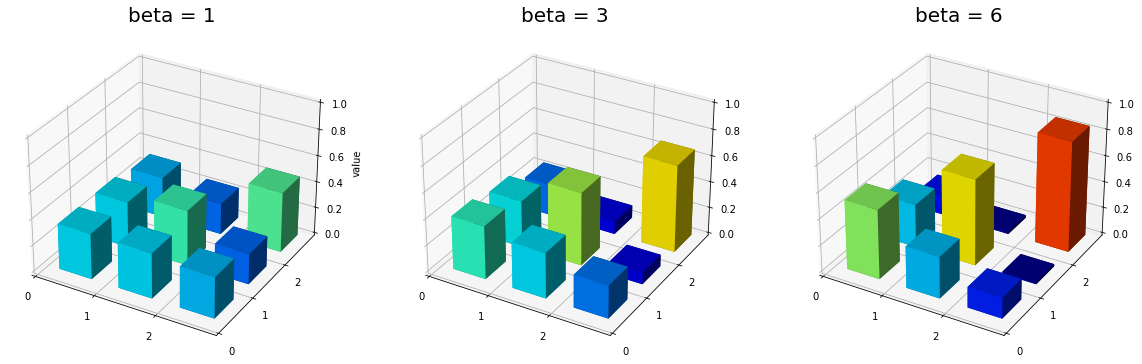}

\end{minipage}
\captionof{figure}{The heights of histograms represent values of corresponding elements in $ {S}^{\beta}_{ {X}}$. As $\beta$ increases, ${S}^{\beta}_{{X}}$ gradually converges towards the solution of the assignment problem, namely, the identity matrix.}
\label{fig_example}
\end{center}

\section{Adaptive softassign} 

This section introduces an adaptive softassign algorithm and some nice Sinkhorn operation rules. 

\subsection{Adaptive softassign}
The performance of the softassign depends on the parameter $\beta$: a larger $\beta$ leads to a better result but more Sinkhorn iterations \cite{cuturi2013sinkhorn}. Theoretically, $ {S}^{\infty}_X$ is the optimal solution for the problem \eqref{eq.projection} \cite{cominetti1994asymptotic}, while the corresponding time cost is exorbitantly high. Therefore, we aim to design an adaptive softassign that can automatically select a moderately sized $\beta$ while still yielding promising results for various large graph matching problems. Inspired by the analysis of optimal transport problems \cite{luise2018differential}, we analyze the relation between $\beta$ and optimal score to provide feasibility for the aim.
\begin{proposition}
    For a square matrix $ {X}$ and $\beta>0$, we have
    
    \begin{equation}
    \begin{aligned}
               | \langle  {S}^{\beta}_{ {X}},  {X} \rangle - \langle  {S}^{\infty}_X,  {X}\rangle | &\leq \| {S}^{\beta}_X - {S}^{\infty}_X\| \| {X}\| 
               \\\| {S}^{\beta}_X - {S}^{\infty}_X\|
               &\leq \frac{c}{\mu} (e^{(-\mu \beta)}),
    \end{aligned}
        \label{eq: assign_marginal}
    \end{equation}
where $c$ and $\mu>0$  are  constants independent of $\beta$.
\end{proposition}
Proposition 1 illustrates an exponential decay of $| \langle  {S}^{\beta}_{ {X}},  {X} \rangle - \langle  {S}^{\infty}_X,  {X}\rangle |$ with respect to $\beta$. This Proposition supports that a moderately sized $\beta$ can yield favorable outcomes. Such a $\beta$ can be determined by setting a threshold of $\| {S}^{\beta}_{ {X}} - {S}^{\infty}_{ {X}}\|$, which is a trade-off between accuracy and efficiency. However, $ {S}^{\infty}_{ {X}}$ is unknown, so we consider utilizing $\| {S}^{\beta}_{ {X}} - {S}^{\beta + \Delta \beta}_{ {X}}\|$ to determine $\beta$.  Then we analyze the convergence of $\| {S}^{\beta}_{ {X}} - {S}^{\beta + \Delta \beta}_{ {X}}\|$.

\begin{proposition}
For a square matrix $ {X}$ and  $\beta,\  \Delta \beta > 0$, we have
    \begin{equation}
        \| {S}^{\beta}_{ {X}} - {S}^{\beta+\Delta \beta}_{ {X}} \|\leq (1- e^{(-\mu \Delta \beta)}) \frac{c}{\mu} e^{(-\mu \beta)} , 
    \end{equation}
     where $c$ and $\mu>0$ are constants independent of $\beta$.
\end{proposition}
Proposition 2 indicates that $\| {S}^{\beta}_{ {X}} -  {S}^{\beta + \Delta \beta}_{ {X}}\|$ and $\| {S}^{\beta}_{ {X}} -  {S}^{\infty}_{ {X}}\|$ decay at similar order as $ \beta$ increases. This allows us to use $\beta +\Delta \beta $ instead of $\infty$ to choose a suboptimal $\beta_\epsilon$:
\begin{equation}
    \beta_{\epsilon} = \arg \min_{\beta} \beta, \ \ \ s.t. \ \| {S}^{\beta}_{ {X}} -  {S}^{\beta + \Delta \beta}_{ {X}}\| \leq \epsilon. \label{eq:beta_epsilon}
\end{equation}
% which supports that a moderately sized $\beta$ can yield favorable outcomes.
The pseudocode for adaptive softassign appears in the Algorithm \ref{ag: adaptive softassign}.

\begin{figure}
    \centering
    \includegraphics[width=0.45\textwidth]{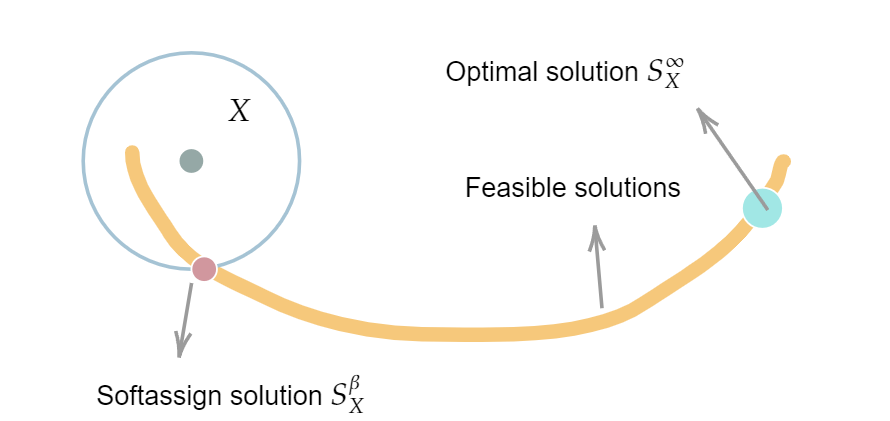}
    \includegraphics[width=0.45\textwidth]{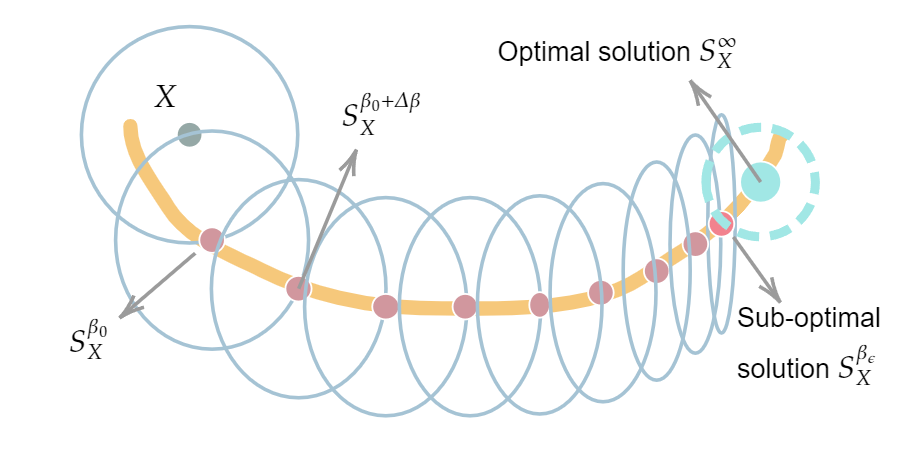}
    \caption{Softassign and adaptive softassign process.}
    \label{fig:Adaptive softassign}
\end{figure}

\textbf{On choosing of $\Delta \beta$}
\ \ \ \citet{altschuler2017near} and \citet{shen2022dyna} demonstrate that softassign can be robust to the nodes' cardinality $n$ by setting $\beta = \gamma \ln(n)$, where $\gamma$ is a constant related with the type of matching graphs. Enlightened by this, we also set $\Delta \beta = \ln(n)$ so that the adaptive softassign is robust to $n$.

\textbf{On choosing of $\beta_0$}
\ Empirical evidence suggests that the computational time required for adaptive softassign positively correlates with $|\beta_\epsilon -\beta_0|$. Therefore, choosing a $\beta_0$,  close to $\beta_{\epsilon}$, can enhance the algorithm's efficiency. The choice of $\beta_0$ for graph matching is discussed in detail in the subsequent section, as $\beta_{\epsilon}$ varies across different problems.

\textbf{Error analysis} \  Since the adaptive softassign has the same accuracy as softassign with $\beta_{\epsilon}$, the performance of adaptive softassign is guaranteed by \cite{shen2022dyna}: 
\begin{equation}
    \frac{1}{n}|\langle S_X^{\beta_{\epsilon}}, X\rangle-\left\langle S_X^{\infty}, X\right\rangle| \leq \frac{\ln(n)}{\beta_{\epsilon}} =\frac{1}{\gamma_{\epsilon}},
\end{equation}
where the left term, an \textit{average assignment error},  quantifying the distance between $S_X^{\beta_{\epsilon}}$ and the optimal solution $S_X^{\infty}$.

\begin{algorithm}[H]
    \centering
    \caption{Adaptive softassign}\label{alg: AS}
    \footnotesize
    \begin{algorithmic}[1]
            \Require $ {X}, \beta_0, \epsilon$     
                 % \State $Initial$  $\beta = \beta_0$
                 \State Compute $ {S}_{ {X}}^{\beta_0}$ by softassign in \eqref{eq:softassign_infla} and \eqref{eq:softassign_sinkhorn} 
                 \For{$k=0,1,2 \dots, $ until $r<\epsilon$ } 
                    \State $\beta_k = \beta_{k-1}+\Delta \beta$
                    \State Compute $ {S}_{ {X}}^{\beta_k}\ \ \ $ (Accelerated by Alg. 2)
                    \State $r = \| {S}_{ {X}}^{\beta_k} -  {S}_{{X}}^{\beta_{k-1}}\|_1$
                \EndFor
                \State \textbf{Return} {$ {S}_{ {X}}^{\beta_k},\  \beta_{k}$}
    \end{algorithmic}
   \label{ag: adaptive softassign}
\end{algorithm}

\begin{algorithm}[H]
    \centering
    \caption{Softassign Transition}\label{algorithm1}
    \footnotesize
    \begin{algorithmic}[1]
    \Require $ {S}_{}^{\beta_{k-1}}, \beta_{k-1},\beta_{k}$
        %\State $\gamma_{(k)} = \gamma_{(k-1)} +1$
        \State $\hat{ {S}}=( {S}_{}^{\beta_{k-1}})^{\circ(\frac{\beta_{k}}{\beta_{k-1}})}$
        \For{$\ell=0,1,2 \dots, $ until convergence}
        \State $ \mathbf{r}^{(\ell+1)} \stackrel{\text {  }}{=}  \mathbf{1} \oslash \hat{ {S}}  \mathbf{c}^{(\ell)}$ 
        \State $ \mathbf{c}^{(\ell+1)} =  \mathbf{1} \oslash \hat{ {S}}^T  \mathbf{r}^{(\ell+1)}$
        \EndFor
        \State \textbf{Return}  $ {S}^{\beta_{k}} =  {D}_{( \mathbf{r})}\hat{ {S}} {D}_{( \mathbf{c})}$
    \end{algorithmic}
    \label{ag.transition}
\end{algorithm}
\subsection{Softassign Transition}

Since the adaptive softassign inevitably compute $ {S}^{\beta+ \Delta \beta}_{X}$ for different $\beta$ repeatedly, we propose a delicate strategy to compute $ {S}^{\beta + \Delta \beta}_{ X}$ from $ {S}^{\beta}_{ X}$ instead of $ X \in \mathbb{R}^{n \times n}$. This recursive computation is much easier than direct computation. The process is shown in Figure \ref{fig:Adaptive softassign}.

To achieve the recursive computation, we first propose some nice Sinkhorn formulas. For convenience, we use $\mathcal{P}_{sk}{( {X})}$ to represent $Sinkhorn( {X}) = {D}_{( \mathbf{r})} {X} {D}_{( \mathbf{c})}$ where $ \mathbf{r}$ and $ \mathbf{c} \in \mathbb{R}^n_+$ are balancing vectors resulting from \eqref{eq:softassign_sinkhorn}.

%propostion
\begin{proposition} {Hadamard-Equipped Sinkhorn}\\
Let $ {X} \in \mathbb{R}^{n \times n}_+$, then
\begin{equation}
    \mathcal{P}_{sk}( {X}) ={X} \circ  {SK}^{( {X})}= {X} \circ ( \mathbf{r}\otimes  \mathbf{c}^T)
\end{equation}
where $ {SK}^{( {X})} \in \mathbb{R}^{n \times n}$  is unique, $ \mathbf{r}$ and $ \mathbf{c} \in \mathbb{R}^n_+$ are balancing vectors so that ${D}_{( \mathbf{r})} {X} {D}_{( \mathbf{c})}$ is doubly stochastic. 
\label{pro.hadamard}
\end{proposition}
%is called the Sinkhorn project matrix for $ {X}$ and it
% \begin{proof}
% According to the Sinkhorn theorem \cite{sinkhorn1967concerning}, $\mathcal{P}_{sk}( {X}) =  {D}_{( \mathbf{r})} {X} {D}_{( \mathbf{c})}$ and it is unique,  where $ {D}_{( \mathbf{r})} =\operatorname{diag}( \mathbf{r})$. Then 
% \begin{equation}  {D}_{( \mathbf{r})} {X} {D}_{( \mathbf{c})} =  {X} \circ ( \mathbf{r}^T\otimes  \mathbf{c}) =  {X} \circ  {SK}^{( {X})}.
% \end{equation}
% \end{proof}
This Proposition builds a bridge between the Hadamard product and the Sinkhorn method. The connection yields some convenient Sinkhorn operation rules.
\begin{lemma}
    Let $X \in \mathbb{R}^{n \times n}_+$, $ \mathbf{u}$ and $ \mathbf{v} \in \mathbb{R}^n_+$, then
    \begin{equation}
        \mathcal{P}_{sk}(X) = \mathcal{P}_{sk}(X\circ(\mathbf{u} \otimes  \mathbf{v}^T)).
    \end{equation}
\end{lemma}

\begin{lemma}{Sinkhorn-Hadamard product}\\
Let $ {X}_1,  {X}_2 \in \mathbb{R}^{n \times n}_+$,  then
\begin{equation}
        \mathcal{P}_{sk}( {X}_1 \circ  {X}_2) = \mathcal{P}_{sk}(\mathcal{P}_{sk}( {X}_1) \circ  {X}_2).
\end{equation}

\end{lemma}

\begin{lemma}{Sinkhorn-Hadamard power}\\
Let $ {X} \in \mathbb{R}^{n \times n}_+$,  
then 
\begin{equation}
      \mathcal{P}_{sk}( {X}^{\circ (ab)}) =\mathcal{P}_{sk}(\mathcal{P}_{sk}( {X}^{\circ a})^{\circ b}),  
\end{equation}
where $a$ and $b$ are two constants not equal to zero.
\end{lemma}

According to the Lemma 1 and Lemma 2, we have
\begin{theorem}{Softassign Transition}\\
Let $ X \in \mathbb{R}^{n \times n}_{+}$,  then
    \begin{equation}
         {S}^{\beta_2}_ X = \mathcal{P}_{sk}(( {S}^{\beta_1}_ X)^{\circ( \frac{\beta_2}{\beta_1})}) , \ where \ \beta_1,\beta_2 > 0.
         \label{eq:transition}
    \end{equation}
\end{theorem}
The softassign transition enables us to compute ${S}_ X^{\beta+\Delta \beta}$ from $ {S}_ X^{\beta}$, which significantly reduces the computational cost. The strategy is detailed in Algorithm \ref{ag.transition}. Its performance is displayed in Figure \ref{fig: Softassign pro}. When the matrix size is 2000, the speedup ratio of the strategy is 6.7x.

\begin{figure}[h]
    \centering
    \includegraphics[width=0.35\textwidth]{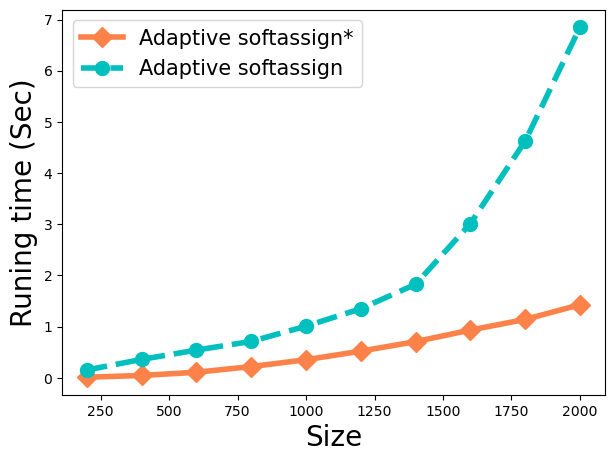}
    \caption{The orange solid line represents the performance of adaptive softassign; the blue dashed line represents the performance of adaptive softassign* (adaptive softassign with the softassign transition). These two methods are evaluated on random matrices over 20 runs.}
    \label{fig: Softassign pro}
\end{figure}

\subsection{Stability}
% In this subsection, we explain how adaptive softassign avoids instability and propose an inexact-exact softassign to compute $ robustly $ S^{\beta}_{X}$.
For a large $\beta$, the computation of softassign may cause  
numerical instability. The instability includes (1) overflow: the elements of $J$ in \eqref{eq:softassign_exp} are too large to handle, and (2) underflow: a row/column sum of $\hat{J}$ approaches to 0 in \eqref{eq:softassign_infla}, then a denominator of zero occurs in the Sinkhorn process \eqref{eq:softassign_sinkhorn} \cite{xie2020fast}. Adaptive softassign can significantly reduce such a risk since it calculates the $S^{\beta_\epsilon}_X$ by $S^{\beta_0}_X$ and a series of softassign transitions. It is akin to dividing a vast distance into smaller segments, thereby enabling one to traverse the distance step by step.

\begin{example} How adaptive softassign avoids instability when finding $S_X^8$ for 
\begin{equation}
    {X}=\left(\begin{array}{cc}
	 -99&  -100 \\
	-100 & -99  
	\end{array}\right).
 \end{equation}
Calculating the $S_X^{8}$ directly will cause instability: $\exp(-99\times 8)$ and $\exp(-100\times 8)$ are smaller than the smallest number that a program can handle, so the program rounds down $\exp(8X)$ to a zero matrix.

A stable choice is computing it by a two-step computation:
\begin{equation}
    S_X^{8} = \mathcal{P}_{sk}((S_X^{2})^{\circ 4})\  or \ \mathcal{P}_{sk}((S_X^{4})^{\circ 2}).
\end{equation}
We show the results as follows:
$$
    S_X^{4} = \left(\begin{array}{cc}
	0.98 & 0.02  \\
	0.02 & 0.98  
	\end{array}\right), \mathcal{P}_{sk}((S_X^{4})^{\circ 2})= \left(\begin{array}{cc}
	0.997 & 0.003  \\
	0.003 & 0.997  
	\end{array}\right)
$$
$$
    S_X^{2} = \left(\begin{array}{cc}
	0.88 & 0.12 \\
	0.12& 0.88 
	\end{array}\right), \mathcal{P}_{sk}((S_X^{2})^{\circ 4})= \left(\begin{array}{cc}
	0.997 & 0.003  \\
	0.003 & 0.997  
	\end{array}\right)
$$
The results also validate the softassign transition that $\mathcal{P}_{sk}((S_X^{2})^{\circ 4})=\mathcal{P}_{sk}((S_X^{4})^{\circ 2}).$ The risk of overflow can also be addressed by this method.
\end{example}

\subsection{Connection with the optimal transport problem}
The Sinkhorn formulas, introduced in Section 3.2, are closely related to the optimal transport problem.
\citet{cuturi2013sinkhorn} formulates the regularized optimal transport problem as 
\begin{equation}
    \mathrm{T}_{ {C}}^{\beta}(\mathbf{a}, \mathbf{b}) {=} \arg \min _{ {T} \in \mathcal{U}(\mathbf{a}, \mathbf{b})}\langle {T},  {C}\rangle- \frac{1}{\beta} \mathcal{H}( {T}), 
    \label{eq:OT}
\end{equation}
where $\mathcal{U}_{(\mathbf{a}, \mathbf{b})}:=\left\{ {T} \in \mathbb{R}_{+}^{n \times n}:  {T} \mathbf{1}=\mathbf{a},  {T}^T \mathbf{1}=\mathbf{b}\right\}$, $ {C} \in \mathbb{R}_{+}^{n \times n}$ is a given cost matrix, and $\mathbf{a},\ \mathbf{b} \in \mathbb{R}_+^{n}$ are given vectors with positive entries with the sum being one. The regularized linear assignment problem \eqref{eq: entropic assignment} is a special case of the regularized optimal transport problem where $\mathbf{a}$ and $\mathbf{b}$ are vectors of ones. The solution of \eqref{eq:OT} has the form 
\begin{equation}
    \mathrm{T}_{ {C}}^{\beta}(\mathbf{a}, \mathbf{b}) =  {D}_{(\mathbf{u} )} \exp(-\beta \mathbf{C})  {D}_{(\mathbf{v} )},
    \label{eq:OT_solution}
\end{equation}
where $\mathbf{v}$ and $\mathbf{u}$ can be  computed by the Sinkhorn iteration. The form of \eqref{eq:OT_solution} is very similar to the solution of the regularized assignment problem in \eqref{eq:softassign_matrix}. According to Proposition \ref{pro.hadamard}, we have $\mathrm{T}_{ {C}}^{\beta}(\mathbf{a}, \mathbf{b}) =\exp(-\beta  {C}) \circ (\mathbf{u} \otimes \mathbf{v}^T)$, and the $\mathbf{u}$ and $\mathbf{v}$ are unique \cite{cuturi2013sinkhorn}. This property makes it easy to prove Lemma 1, Lemma 2, Lemma 3, and the transition theorem for optimal transport problems. Such theoretical results will provide more flexibility for computation and shed light on optimal transport problems. For instance, \citet{liao2022fast1,liao2022fast2} enhance the Sinkhorn method in special optimal transport problems by leveraging Hadamard operations (which differs from our Sinkhorn formulas). Another interesting finding based on the Sinkhorn formulas is that adaptive softassign is a variant of the proximal point method for optimal transport problems (described in the Appendix).

\section{The adaptive softassign matching algorithm}

%Typically, we set $\beta = 2\Delta \beta = 2\ln \tilde{n}$ as the initial value.
The adaptive softassign matching algorithm\footnote{Our codes are available at https://github.com/BinruiShen/Adaptive-Softassign-Matching.} is shown in Algorithm \ref{ag.asm}.  In step 1, a uniform initialization approach is adopted when no prior information is available. For the problem of matching graphs with different numbers of nodes (we assume that $\tilde{n} \leq n$), \citet{gold1996graduated} introduce a square \textit{slack matrix} like $ \hat{D}$ in step 5: $ \hat{D}_{(1:n,1:\tilde{n})} = {A} {N} \tilde{ {A}}+\lambda  {K}$ and rest elements of $ \hat{D}$ are zero. Discussion in \cite{lu2016fast} indicates that matching quality is not sensitive to the parameter $\lambda$, and we set $\lambda=1$ follows \cite{lu2016fast}. In step 7, we utilize $\beta^{(k)}_{\epsilon}- \Delta \beta$ as the $\beta^{(k+1)}_0$ to reduce the computational cost of the adaptive softassign in the next iterate: such a $\beta^{(k+1)}_0$ is close to $\beta^{(k+1)}_{\epsilon}$, since $\beta_\epsilon$ typically increases in early iterations of the algorithm before leveling off at a certain point with minor fluctuations (see Figure \ref{fig:gamma_change}). It should be noted that $\beta^{(k)}_{\epsilon} \geq \beta^{(k)}_0 + \Delta \beta$ according to Algorithm \ref{ag: adaptive softassign}, which indicates that  $\{\beta^{(k)}\}$ will inevitably be an increasing series if $\beta^{(k+1)}_0 = \beta^{(k)}_{\epsilon}$. The discretization in Step 10 is completed by the Hungarian method \cite{kuhn1955hungarian}. 
\begin{figure}
    \centering
    \includegraphics[width=0.95\linewidth]{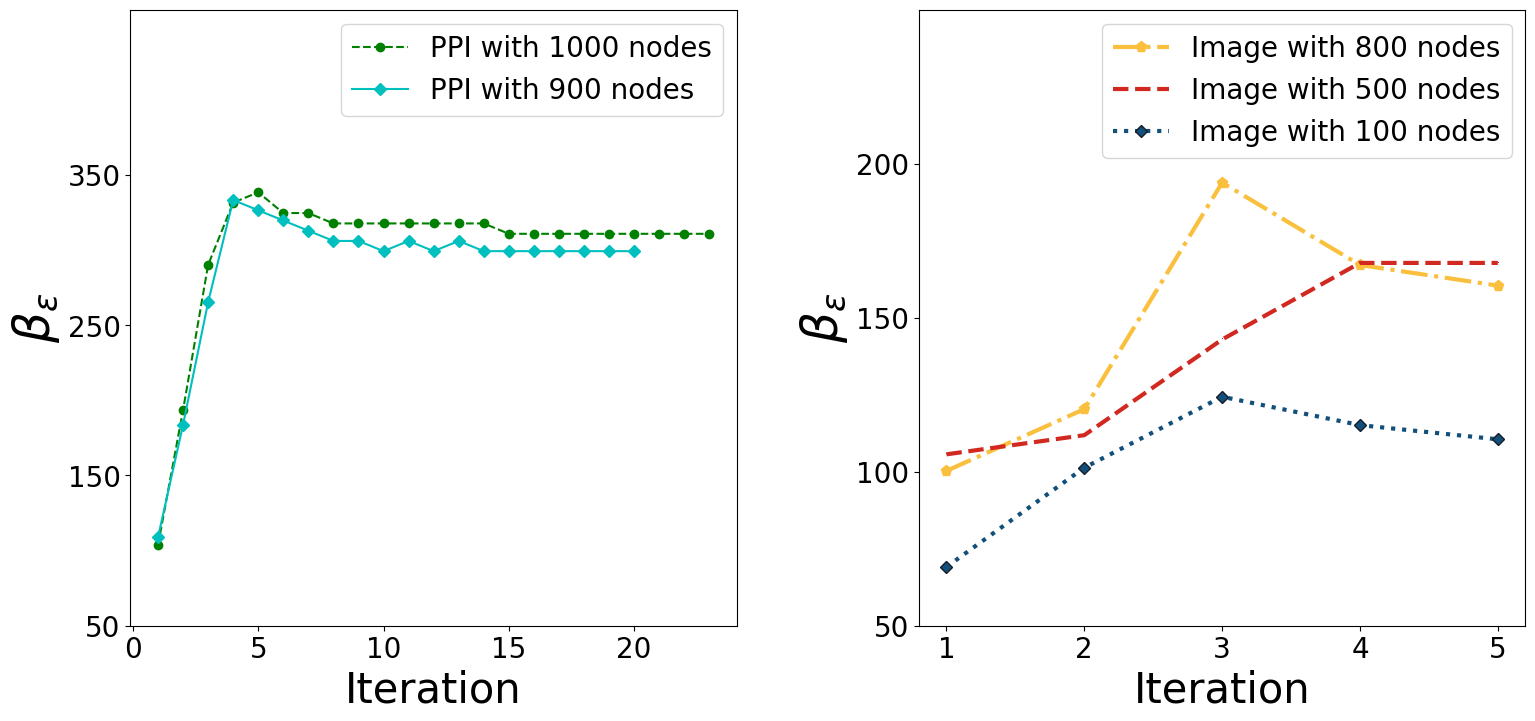}
    \caption{The change of $\beta_\epsilon$ in ASM when $\beta_0$  is $\ln n$ in adaptive softassign. PPI and image are two kinds of graph matching tasks introduced in experiments.}
    \label{fig:gamma_change}
\end{figure}

Regardless of fast and sparse matrix computation, step 4 and step 5 entail $O(n^3)$ operations per iteration. In the matching process, $\beta_{\epsilon}$ in the adaptive softassign (step 6) includes an increasing and stable state, shown in Figure \ref{fig:gamma_change}. In stable state, the cost of adaptive softassign is close to that of softassign with $\beta_{\epsilon}^m$, where $\beta_{\epsilon}^m$ is the maximum of $\beta_{\epsilon}$ in the matching process. In the increasing state, the cost of adaptive softassign is less than that of softassign with $\beta_{\epsilon}^m$. Therefore, the average cost of adaptive softassign in a matching process is close to the softassign with $\beta_{\epsilon}^m$: $O({n^2\beta_{\epsilon}^m\|{X}\|_{\infty}})$  where the maximum of ${X}$ is 1 \cite{luo2023improved}. Step 8 requires $O(n^2)$ operations per iteration. The Hungarian algorithm completes the final discretization step \cite{kuhn1955hungarian} with complexities of $O(n^3)$. Thus, the algorithm has time complexity $O(n^3)+O({n^2\beta_{\epsilon}^m\|{X}\|_{\infty}})$ per iteration and space complexity $O(n^2)$.

% From an algorithmic perspective, the method is similar to the graduated assignment (GA) \cite{gold1996graduated}, which also adapts the softassign method. However, unlike GA, our proposed method features adaptivity instead of relying on a deterministic annealing schedule with three parameters. It makes the proposed algorithm more convenient, scalable, and robust, as evidenced by the experiments.

\begin{CJK*}{UTF8}{gkai}
    \begin{algorithm}[h]
        \caption{ $\,$ Adaptive softassign matching (ASM)}
        \begin{algorithmic}[1] 
            \Require $ {A},\tilde{ {A}}, {K},\lambda$
            \Ensure $ {M}$       
                 \State Initial  $ \Delta\beta = \ln n, {N}^{(0)} =  {(\mathbf{\frac{1}{n})}}_{n \times \tilde{n}},  \hat{D}^{(0)}= \mathbf{0}_{n \times n}$ 
                 \State $\beta^{({1})}_0 =\Delta\beta$
                 \For{$k=1,2 \dots, $ until ${N}$ converge}
                 \State Compute optimal $\alpha$ by \eqref{eq: adaptive alpha}     
               \State $ \hat{{D}}^{(k)}_{(1:n,1:\tilde{n})} = {A} {N}^{(k-1)} \tilde{ {A}}+\lambda  {K}$
                 \State $[ {D}^{(k)},\  \beta_{\epsilon}^{(k)}]=$ Adaptive  softassign$( \hat{D}^{(k)},\ \beta^{(k)}_0)$

                 \State $\beta^{(k+1)}_0 = \beta_{\epsilon}^{(k)} - \Delta\beta$
                    \State $ {N}^{(k)} =(1-\alpha) {N}^{(k-1)}+\alpha  {D}^{(k)}_{(1:n,1:\tilde{n})}$
                \EndFor
                \State Discretize ${N}$ to ${M}$
                \State \Return{$ {M}$}
        \end{algorithmic}
        \label{ag.asm}
    \end{algorithm}
\end{CJK*}

\section{Experiments}
\textbf{Baselines} We compare ASM against the following baselines: DSPFP \cite{lu2016fast}, GA \cite{gold1996graduated}, AIPFP \cite{leordeanu2009integer,lu2016fast}, SCG \cite{shen2022dyna}, GWL\footnote{https://github.com/HongtengXu/gwl} \cite{xu2019gromov} , S-GWL\footnote{https://github.com/HongtengXu/s-gwl} \cite{xu2019scalable}, MAGNA++\footnote{https://www3.nd.edu/~cone/MAGNA++/} \cite{vijayan2015magna++}, and GRASP\footnote{https://github.com/AU-DIS/GRASP} \cite{hermanns2023grasp}.

\noindent \textbf{Benchmarks} We perform algorithms in three benchmarks: the protein-protein interaction network (PPI), Facebook social networks, and real images. Unweighted graphs represent the first two networks. Weighted graphs with attributed nodes are extracted from real images. 

\noindent \textbf{Evaluations}
The evaluation in PPI and the social network is node accuracy $\frac{n_c}{n}$ where $n_c$ represents the number of correct matching nodes. Since the ground truth of matching on real images is unknown, we evaluate the algorithms by matching error 
\begin{equation}
    \frac{1}{4}\left\| {A}- {M}\widetilde{ {A}} {M}^{T}\right\|_{Fro}^{2}+\left\| {F}- {M} \widetilde{ {F}}\right\|_{Fro}^{2}.
    \label{eq:error_fun}
\end{equation}
The first four baselines can adapt the \eqref{eq:error_fun} as the objective function. Other algorithms are not compared in real image experiments since they are not designed to solve matching problems with attributed nodes. 
\begin{figure*}[h]
    \centering
    \includegraphics[width=0.9\textwidth]{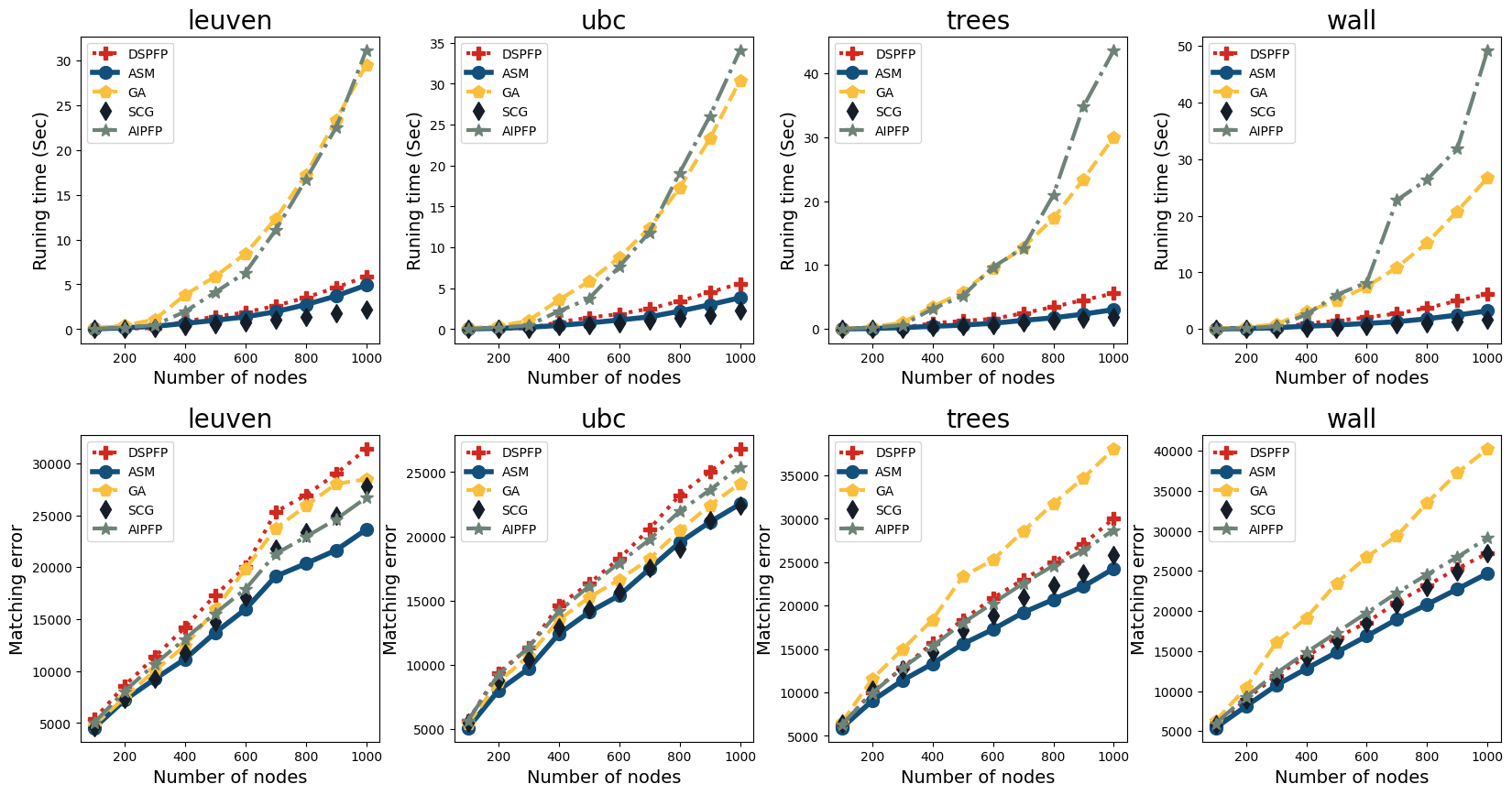}{}

    \caption{Comparision between algorithms in four graph pairs.}
    \label{fig: real image}
\end{figure*}
% Its noisy version contains $q\%$ more low-confidence interactions, and $ q \in \{5, 10, 15, 20, 25\}$.
\subsection{Protein network and Social network}
The yeast's protein-protein interaction (PPI) networks contain 1,004 proteins and 4,920 high-confidence interactions\footnotemark[4]. The social network comprising 'circles' (or 'friends lists') from Facebook \cite{snapnets} contains 4039 users (nodes) and 88234 relations (edges). Following the experimental protocol of \cite{xu2019scalable}, we compare different methods on matching networks with $5\%$, $15\%$ and $25\%$ noisy versions. Table \ref{tab:PPI} and Table \ref{tab:facebook} list the performance of various methods. The ASM consistently attains the highest accuracy across all scenarios, demonstrating its robustness. Notably, it yields an approximate $20\%$ enhancement in accuracy amidst a $25\%$ noise level, further accentuating its efficacy. Compared to the suboptimal algorithm GWL, ASM showcases an efficiency improvement of approximately tenfold.

% It achieves nearly $20\%$ accuracy improvement 
% and its advantages are particularly prominent in noisy environments.
\begin{table}[ht]
\centering
\caption{Comparisons on yeast PPI}
\label{tab:PPI}
\resizebox{\columnwidth}{!}{%
\begin{tabular}{c|ccccccc}
\hline
Yeast network & \multicolumn{2}{c}{5\% noise} & \multicolumn{2}{c}{15\% noise} & \multicolumn{2}{c}{25\% noise} \\ \hline
Methods & Node Acc & time  & Node Acc & time  & Node Acc & time  \\ \hline
MAGNA++ \cite{vijayan2015magna++} & 48.3\%   & 603.3s   & 25.0\%  & 630.6s   & 13.6\%  & 624.2s   \\
S-GWL \cite{xu2019scalable}   & 81.3\%   & 82.3s    & 62.4\%   & 82.1s    & 55.5\%  & 88.4s   \\
GWL\cite{xu2019gromov}   & 83.7\%   & 226.4s    & 66.3\%   & 254.7s    & 57.6\%  & 246.5s   \\
DSPFP \cite{lu2016fast}   & 78.1\%   & 10.2s    & 60.8\%   & 
10.14s     & 42.9\%   & 9.8s       \\
GA \cite{gold1996graduated}      & 14.0\%   & 24.4s    & 9.6\%    & 24.5s    & 7.4\%    & 24.0s    \\

GRASP \cite{hermanns2023grasp}     & 38.6\%   & \textbf{1.1}s    & 8.3\%   & \textbf{1.2}s    & 5.6\%   & \textbf{1.2}s    \\
SCG \cite{shen2022dyna}      & 73.1\%   & 10.7s    & 53.1\%    & 10.3s    & 43.0\%    & 10.0s    \\

AIPFP \cite{leordeanu2009integer,lu2016fast}     & 43.1\%   & 105.4s    & 27.1\%   & 75.2s   & 22.1\%   & 73.8s\\ \hline
ASM     & \textbf{89.0\%}   & 28.7s    & \textbf{81.2\%}   & 22.7s    &\textbf{ 75.1\% }  & 22.6s    \\ \hline
\end{tabular}%
}
\end{table}

% \subsection{Social network}
%  Following the experimental protocol of \cite{xu2019scalable}, we compare different methods on matching social networks with $5\%$, $15\%$ and $25\%$ noisy versions. Table \ref{tab:facebook} lists the performance of various methods in this task. 

% The dataset includes node features (profiles), circles, and ego networks.
\begin{table}[ht]
\centering
\caption{Comparisons on Facebook network}
\label{tab:facebook}
\resizebox{\columnwidth}{!}{%
\begin{tabular}{c|ccccccc}
\hline
Social network & \multicolumn{2}{c}{5\% noise} & \multicolumn{2}{c}{15\% noise} & \multicolumn{2}{c}{25\% noise} \\ \hline
Methods & Node Acc & time  & Node Acc & time  & Node Acc & time  \\ \hline
S-GWL \cite{xu2019scalable}   &  26.4\%   &  1204.1s    &  18.3\%   &  1268.2s    &  17.9\%  &  1295.8s   \\
GWL\cite{xu2019gromov}   &  78.1\%   &  3721.6s    &  68.4\%   &  4271.3s    &  60.8\%  &  4453.9s   \\
DSPFP \cite{lu2016fast}   &  79.7\%   & 151.3s    & 68.3\%   & 
154.2s     & 62.2\%   & 156.9s       \\
GA \cite{gold1996graduated}      & 35.5\%   & 793.2s    & 21.4\%    & 761.7s    & 16.0\%    & 832.6s    \\

GRASP\cite{hermanns2023grasp}     & 37.9\%   & \textbf{63.6s}    & 20.3\%   & \textbf{67.4s}    & 15.7\%   & \textbf{71.3s}    
\\ SCG \cite{shen2022dyna}      & 58.2\%   & 211.7s    & 43.1\%    & 221.3s    & 43.1\%    & 211.0s    \\

AIPFP \cite{leordeanu2009integer,lu2016fast}     & 68.6\%   & 2705.5s    & 55.1\%   &2552.7s   & 47.8\%   & 2513.8s    \\ \hline
ASM     & \textbf{91.1\%}   & 387.2s    & \textbf{88.4\%}   & 391.7s    &\textbf{ 85.7\% }  & 393.1s    \\ \hline
\end{tabular}%
}
\end{table}

\subsection{Real images}
\begin{table*}[h]
\centering
\caption{Comparisons for graph matching methods on real images with 1000 nodes}
\label{tab: real image}
\resizebox{1.8\columnwidth}{!}{%
\begin{tabular}{c|cc|cc|cc|cc}
\hline
Image set & \multicolumn{2}{c|}{Leuven}      & \multicolumn{2}{c|}{ubc}         & \multicolumn{2}{c|}{trees}       & \multicolumn{2}{c}{wall}         \\ \hline
Methods   & Error ($\times 10^4$) & time (s) & Error ($\times 10^4$) & time & Error ($\times 10^4$) & time & Error ($\times 10^4$) & time \\
DSPFP \cite{lu2016fast}& 3.1 & 5.9s  & 2.7 & 5.6s  & 3.0 & 5.6s  & 2.7 & 6.1s  \\
GA \cite{gold1996graduated}   & 2.8 & 29.5s & 2.4 & 30.3s & 3.8 & 30.0s & 4.0 & 26.7s \\
AIPFP \cite{leordeanu2009integer,lu2016fast} & 2.7 & 31.1s  & 2.5 & 34.1s  & 2.9 & 43.6s  & 2.9 & 49.2s  \\
SCG \cite{shen2022dyna}   & 2.7 & \textbf{1.7}s & 2.3 & \textbf{1.8}s & 2.5 & \textbf{1.5}s & 2.6 & \textbf{1.3}s \\
ASM   & \textbf{2.3} & 4.9s  & \textbf{2.2} & 3.8s  & \textbf{2.4} &3.0s  & \textbf{2.5} &3.2s  \\ \hline
\end{tabular}%
}
\end{table*}
In this set of experiments, we construct attributed weighted graphs from a public dataset\footnote{http://www.robots.ox.ac.uk/~vgg/research/affine/}, which covers five common picture transformations: viewpoint changes, scale changes, image blur, JPEG compression, and illumination. 

Following the experimental protocol of \cite{lu2016fast}, the construction includes extraction of nodes, selection of nodes, and calculation of edge weight. We extract key points by scale-invariant feature transform (SIFT) as candidates of nodes, and corresponding feature vectors are also obtained in this step. Nodes are selected if the node candidates have high similarity (inner product of feature vectors) with all candidate nodes from another graph. Then, all chosen nodes are connected, and the weights of edges are measured by the Euclidean distance between two corresponding nodes. 

The running time and matching error are calculated by the average results of five matching pairs (1 vs. 2, 2 vs. 3, 3 vs. 4, 4 vs. 5, 5 vs. 6) from the same picture set. The results are shown in Figure \ref{fig: real image}. More details on experiments with 1000 nodes are shown in Table \ref{tab: real image} for further comparison. The ASM method consistently attains the lowest error across all cases while maintaining comparable efficiency. 
\section{Conclusion}
This paper proposes an adaptive softassign method for large graph matching problems. It can automatically tune the parameter according to a given error bound, which is convenient and robust. The resulting matching algorithm enjoys significantly higher accuracy than previous state-of-the-art large graph matching algorithms. 
% The proposed algorithm achieved a $37.6\%$ reduction in error and a 7.8x speedup compared to the classical softassign-based algorithm GA on attributed graphs with 1000 nodes.

The proposed Hadamard-Equipped Sinkhorn formulas significantly accelerate the adaptive softassign process and avoid numerical instability in Sinkhorn. These formulas provide a new perspective on operations related to Sinkhorn and optimal transport problems. The Hadamard-Equipped Sinkhorn formulas seem to have some nice properties of group, which might be a promising research direction. 
% Based on this, we discover that adaptive softassign can be considered a variant of the proximal point method but more flexible. A close relationship with

Experiments show that ASM has comparable efficiency in attributed graph matching tasks while the efficiency is not in the first tier in plain graph matching. Therefore, increasing the efficiency in plain graph matching is one of the future works.

\section*{Acknowledgement}
\thanks{The authors would like to appreciate the support from the Interdisciplinary Intelligence Super Computer Center of Beijing Normal University at Zhuhai. This work was partially supported by the Natural Science Foundation of China (12271047); UIC research grant (R0400001-22; UICR0400008-21; UICR04202405-21); Guangdong College Enhancement and Innovation Program (2021ZDZX1046); Key Programme Special Fund in XJTLU (KSF-E-32), Research Enhancement Fund of XJTLU (REF-18-01-04); Guangdong Provincial Key Laboratory of Interdisciplinary Research and Application for Data Science, BNU-HKBU United International College (2022B1212010006).}

\clearpage

{
    \small
    \bibliographystyle{ieeenat_fullname}
    \bibliography{main}
}

% \appendix

\clearpage
\appendix
\maketitlesupplementary
\setcounter{lemma}{0} % 重置引理计数器
\setcounter{proposition}{0}
\section{More details of the projected fixed point method}
Consider the objective function 
\begin{equation}
    \mathcal{Z} ( {M}) = \frac{1}{2} \operatorname{tr}\left({ {M}}^{T} { {A}} { {M}} {\widetilde{ {A}}}\right)+\lambda \operatorname{tr}\left({ {M}}^{T}{ {K}}\right).   
    \label{eq.zm_app}
\end{equation}
Given an initial condition $ {M}^{(k)}$, we can linearize the objective function at $ {M}^{(k)}$ via the Taylor series approximation:
\begin{equation}
    \mathcal{Z}( {M}) \approx \mathcal{Z}( {M}^{(k)}) + \operatorname{tr} \left\{\nabla \mathcal{Z}( {M}^{(k)})^T( {M}- {M}^{(k)})\right\} \label{eq:zml}
\end{equation}
One can find an approximate solution to \eqref{eq.zm_app} by maximizing a sequence of the linearization of $\mathcal{Z}(M)$ in \eqref{eq:zml}. Since $ {M}^{{(k)}}$ is a constant, the maximization of the linear approximation is equivalent to the following linear assignment problem 
\begin{equation}
   \max_{ {M} \in \Sigma_{n\times n} }\operatorname{tr} ( \nabla \mathcal{Z}( {M}^{(k)})^T  {M}) =\max_{ {M} \in \Sigma_{n\times n} } \langle\nabla \mathcal{Z}( {M}^{(k)}),   {M}) \rangle, 
    \label{eq: assignment}
\end{equation}
Therefore, the quadratic assignment problem can be transformed into a series of linear assignment problems. This idea is first proposed in \cite{gold1996graduated} with a different objective function. The solution to each linear assignment problem in \eqref{eq: assignment} can be shown as
\begin{equation}
     {M}^{(k)} = \mathcal{P}(\nabla \mathcal{Z}( {M}^{(k-1)})).
\end{equation}
where $P$ is a projection that includes the doubly stochastic projection used in \cite{lu2016fast} and the discrete projection used in \cite{leordeanu2009integer}.
A generation of such an iterative formula  is
\begin{equation}
     {M}^{(k)} = (1- \alpha)  {M}^{(k-1)}  + \alpha \mathcal{P}(\nabla \mathcal{Z}( {M}^{(k-1)})).
\end{equation}
Such a formula can cover all the points between two doubly stochastic matrices. 
The resulting algorithm is called the \textit{projected fixed-point method} \cite{leordeanu2009integer}.  

% Pseudocode for the adaptive projected fixed-point framework is shown in Algorithm \ref{ag.pfp}. 

% \begin{CJK*}{UTF8}{gkai}
%     \begin{algorithm}[H]
%         \centering
%         \caption{ $\,$ The Adaptive Projected Fixed-Point method}
%         \begin{algorithmic}[1] 
%             \Require $ {A},\tilde{ {A}}, {K},\lambda$
%             \Ensure $ {M}$       
%                  \State Initial  $ {N} =  {(\frac{1}{n})}_{n \times \tilde{n}},  {D}= {0}_{n \times n}$
%                  \While{$ {N}$ does not converge}
%                  \State Compute optimal $\alpha$ by \eqref{eq: adaptive alpha}
%                  \State $ {D}_{(1:n,1:\tilde{n})} = {A} {N} \tilde{ {A}}+\lambda  {K}$
%                  \State $ {D}= \mathcal{P}( {D})$
%                     \State $ {N} =(1-\alpha) {N}+\alpha  {D}_{(1:n,1:\tilde{n})}$
%                 \EndWhile
%                 \State Discretize ${N}$ to ${M}$
%                 \State \Return{$ {M}$}

%         \end{algorithmic}
%         \label{ag.pfp}
%     \end{algorithm}
% \end{CJK*}

\section{Proofs of Proposition 1 and Proposition 2}

\begin{proposition}
    For a square matrix $ {X}$ and $\beta>0$, we have
    \begin{equation}
       | \langle  {S}^{\beta}_{ {X}},  {X} \rangle - \langle  {S}^{\infty}_X,  {X}\rangle | \leq \| {S}^{\beta}_X - {S}^{\infty}_X\| \| {X}\| \leq \frac{c}{\mu} (e^{(-\mu \beta)})\| {X}\|,
    \end{equation}
    where $c$ and $\mu>0$  are  constants independent of $\beta$.
\end{proposition}

\begin{proposition}
For a square matrix $ {X}$ and  $\beta,\  \Delta \beta > 0$, we have
    \begin{equation}
        \| {S}^{\beta}_{ {X}} - {S}^{\beta+\Delta \beta}_{ {X}} \|\leq (1- e^{(-\mu \Delta \beta)}) \frac{c}{\mu} e^{(-\mu \beta)} , 
    \end{equation}
     where $c$ and $\mu>0$ are constants independent of $\beta$.
\end{proposition}

\begin{proof}

We first transform the problem \eqref{eq: entropic assignment} into vector form
\begin{equation}
    \langle  {M},  {X} \rangle + \frac{1}{\beta} \mathcal{H}({M}) = \mathbf{m}^T \mathbf{x} - \frac{1}{\beta} \sum \mathbf{m}_{i} \ln(\mathbf{m}_{i}),
\end{equation}
where $\mathbf{m}=\operatorname{vec}{{M}}$. Since $\sum_{i}^{n^2} \mathbf{m}_i =n$, then
\eqref{eq: entropic assignment} is equivalent to the well studied problem \cite{cominetti1994asymptotic}
\begin{equation}
\mathbf{s}^\beta_{\mathbf{x}} = \operatornamewithlimits{argmin}_{\mathbf{m}}  \mathbf{m}^T (-\mathbf{x}) + \frac{1}{\beta} \sum_{i} \mathbf{m}_{i} (\ln(\mathbf{m}_{i})-1) .
\end{equation}
Let $\dot{{S}}^{\beta}_X$ be derivative of ${S}^{\beta}_X$ with respect to $\beta$, according to the proof of \cite[Proposition 5.1]{cominetti1994asymptotic},  $\dot{{S}}^{\beta}_X$ converges towards zero exponentially i.e., there exist a $c_0>0$ and $\mu >0$ such that $$|(\dot{{S}}^{\beta}_X)_{ij}| \leq c_0 e^{-\mu\beta}.$$
   According to the fundamental theory of Calculus,
   \begin{equation}
          \begin{aligned}
               |({S}^{\infty}_X)_{ij} -({S}^{\beta}_X)_{ij}|  &= |\int^\infty_\beta (\dot{{S}}^{\tau}_X)_{ij}d\tau|\leq \int^\infty_\beta |(\dot{{S}}^{\tau}_X)_{ij}|d\tau 
               \\
               &\leq \frac{c_0}{\mu} (e^{(-\mu \beta)}).
   \end{aligned}
   \end{equation}
Similarly, we have 
    \begin{equation}
    \begin{aligned}
                |({S}^{\beta}_X)_{ij} -({S}^{\beta+\Delta \beta}_X)_{ij}|&=|\int^{\beta+\Delta \beta}_\beta (\dot{{S}}^{\tau}_X)_{ij}d\tau| 
                \\
                &\leq \frac{c_0}{\mu} (e^{(-\mu \beta)} - e^{(-\mu (\beta+\Delta \beta))}).
    \end{aligned}
    \end{equation}
The rest of the proof for the two propositions follows easily from this. 
\end{proof}

\section{Proof of Proposition 3}
\begin{proposition} {Hadamard-Equipped Sinkhorn}\\
Let $ {X} \in \mathbb{R}^{n \times n}_+$, then
\begin{equation}
    \mathcal{P}_{sk}( {X}) ={X} \circ  {SK}^{( {X})}= {X} \circ ( \mathbf{r}\otimes  \mathbf{c}^T)
\end{equation}
where $ {SK}^{( {X})} \in \mathbb{R}^{n \times n}$  is unique, $ \mathbf{r}$ and $ \mathbf{c} \in \mathbb{R}^n_+$ are balancing vectors so that ${D}_{( \mathbf{r})} {X} {D}_{( \mathbf{c})}$ is a doubly stochastic matrix. 
\end{proposition}
\begin{proof}
In elementwise form, for each entry $(i,j)$ of the resulting matrix, we have:
\[
\left[ D_{(\mathbf{r})} X D_{(\mathbf{c})} \right]_{ij} = r_i \cdot X_{ij} \cdot c_j
\]

We now observe that $\mathbf{r}\otimes  \mathbf{c}^T$ is a matrix with entries $(\mathbf{r}\otimes  \mathbf{c}^T)_{ij} = r_i c_j$. Hence, we can express the product as:
\[
(D_{(\mathbf{r})} X D_{(\mathbf{c})})_{ij} = X_{ij} \cdot (\mathbf{r}\otimes  \mathbf{c}^T)_{ij}
\]

This is precisely the Hadamard (elementwise) product of $X$ with the matrix $\mathbf{r}\otimes  \mathbf{c}^T$:
\[
D_{(\mathbf{r})} X D_{(\mathbf{c})} = X \circ (\mathbf{r}\otimes  \mathbf{c}^T)
\]

Thus, we conclude:
\[
\mathcal{P}_{\mathrm{sk}}(X) = D_{(\mathbf{r})} X D_{(\mathbf{c})} = X \circ (\mathbf{r}\otimes  \mathbf{c}^T).
\]

\end{proof}

\section{Proofs of Lemma 1, Lemma 2, and Lemma 3}
\begin{lemma}
    Let $X \in \mathbb{R}^{n \times n}_+$, $ \mathbf{u}$ and $ \mathbf{v} \in \mathbb{R}^n_+$, then
    \begin{equation}
        \mathcal{P}_{sk}(X) = \mathcal{P}_{sk}(X\circ(\mathbf{u} \otimes  \mathbf{v}^T)).
    \end{equation}
\end{lemma}

\begin{proof}
Let $Y = X\circ(\mathbf{u}^T \otimes  \mathbf{v})$, we have
\begin{equation}
    \mathcal{P}_{sk}(Y)= Y \circ (\mathbf{r}_Y \otimes \mathbf{c}_Y^T).
\end{equation}
Then
\begin{equation}
        \begin{aligned}
            \mathcal{P}_{sk}(X\circ(\mathbf{u}\otimes  \mathbf{v}))&=X \circ (\mathbf{u}\otimes  \mathbf{v}^T) \circ (\mathbf{r}_Y\otimes  \mathbf{c}_Y^T)\\
            &=X \circ ((\underbrace{\mathbf{u}\circ \mathbf{r}_Y}_{\mathbf{r}_1})\otimes  (\underbrace{\mathbf{v}\circ  \mathbf{c}_Y}_{\mathbf{c}_1})^T)\\
            &=X\circ(\mathbf{r}_1\otimes  \mathbf{c}_1^T)
            \\
            &=\mathcal{P}_{sk}(X).
        \end{aligned}
    \end{equation}
Since $X\circ(\mathbf{r}_1\otimes  \mathbf{c}_1^T)$ is a doubly stochastic matrix, $\mathbf{r}_1\otimes  \mathbf{c}_1^T = {SK^{(X)}}$ according to the Proposition \ref{pro.hadamard}.
\end{proof}

\begin{lemma}{Sinkhorn-Hadamard product}

Let ${X}_1, {X}_2 \in \mathbb{R}^{n \times n}_+$,  then
$
    \mathcal{P}_{sk}({X}_1 \circ {X}_2) = \mathcal{P}_{sk}(\mathcal{P}_{sk}({X}_1) \circ {X}_2).
$
\end{lemma}
\begin{proof}
According to Lemma 1, the right-hand side is
\begin{align}
        \mathcal{P}_{sk}(\overbrace{\mathcal{P}_{sk}({X}_1)} \circ {X}_2) &= \mathcal{P}_{sk}(\overbrace{{X}_1 \circ {SK}^{({X}_1)}}\circ {X}_2)\\
        &=\mathcal{P}_{sk}({X}_1 \circ {X}_2),
\end{align}
which proves this Lemma.

% According to the Hadamard-Equipped Sinkhorn theorem, the left-hand side of the equation is
% \begin{equation}
%     \mathcal{P}_{sk}({X}_1 \circ {X}_2) ={X}_1 \circ {X}_2 \circ {SK}^{({X}_1 \circ {X}_2)}.
% \end{equation}

% The right-hand side is

% \begin{align}
% &\mathcal{P}_{sk}(\overbrace{\mathcal{P}_{sk}({X}_1)} \circ {X}_2) = \mathcal{P}_{sk}(\overbrace{{X}_1 \circ {SK}^{({X}_1)}}\circ {X}_2)\\
% & = {X}_1 \circ {SK}^{({X}_1)}\circ {X}_2 \circ {SK}^{({X}_1 \circ {SK}^{({X}_1)}\circ {X}_2 )}\\
% & = {X}_1 \circ {X}_2 \circ \underbrace{{SK}^{({X}_1)}\circ {SK}^{({X}_1 \circ {SK}^{({X}_1)}\circ {X}_2 )}}_{\text{Sinkhorn project matrix for ${X}_1 \circ {X}_2$}}.
% \end{align}

%  If  ${SK}^{({X}_1)}\circ {SK}^{({X}_1 \circ {SK}^{({X}_1)}\circ {X}_2 )} \neq {SK}^{({X}_1 \circ {X}_2)}$, then there exists two different Sinkhorn project matrix for ${X}_1 \circ {X}_2$, which violates the uniqueness of the Sinkhorn project matrix. So ${SK}^{({X}_1 \circ {X}_2)} = {SK}^{({X}_1)}\circ {SK}^{({X}_1 \circ {SK}^{({X}_1)}\circ {X}_2 )}$, which proves the lemma.
\end{proof}

\ 
\ 

\begin{lemma}{Sinkhorn-Hadamard power}

Let ${X}_1, {X}_2 \in \mathbb{R}^{n \times n}_+$,  
then $
    \mathcal{P}_{sk}({X}^{\circ (ab)}) =\mathcal{P}_{sk}(\mathcal{P}_{sk}({X}^{\circ a})^{\circ b}),
$
where $a$ and $b$ are two constants not equal to zero.
\end{lemma}

\begin{proof}
According to Lemma 1, the right-hand side is
\begin{align}
\mathcal{P}_{sk}(\mathcal{P}_{sk}({X}^{\circ a})^{\circ b}) &= \mathcal{P}_{sk}(({X}^{\circ a}\circ {SK}^{({X}^{\circ a})})^{\circ b})\\
&= \mathcal{P}_{sk}({X}^{\circ (a b)} \circ ({SK}^{({X}^{\circ a})})^{\circ b}) \\
&=\mathcal{P}_{sk}({X}^{\circ (ab)})
\end{align}
which completes the proof.
% \begin{align}
% \mathcal{P}_{sk}(\mathcal{P}_{sk}({X}^{\circ a})^{\circ b}) &= \mathcal{P}_{sk}(({X}^{\circ a}\circ {SK}^{({X}^{\circ a})})^{\circ b})\\
% &= {X}^{\circ (a b)} \circ ({SK}^{({X}^{\circ a})})^{\circ b} \\
% &\circ {SK}^{({X}^{\circ (a b)} \circ ({SK}^{({X}^{\circ a})})^{\circ b} )}
% \end{align}

% Due to the uniqueness of the Sinkhorn project matrix, we have 

% $$
% ({SK}^{({X}^{\circ a})})^{\circ b} \circ {SK}^{({X}^{\circ (a b)} \circ ({SK}^{({X}^{\circ a})})^{\circ b} )} = {SK}^{({X}^{\circ (a  b)})},
% $$
\end{proof}

% \section{Sinkhorn formulas in optimal transport problems}
% By adapting Sinkhorn formulas, we propose an approximate inexact proximal point method from the inexact proximal point method \cite{xie2020fast} (IPOT). The IPOT 
% \begin{equation}
%      \hat{S}^{(k)}_X = \mathcal{\hat{P}}_{sk}( \hat{S}^{(k-1)}_X \circ \exp(\Delta \beta X))
% \end{equation}
% where $\hat{P}_{sk}$ represents an inexact Sinkhorn method with a very small amount of Sinkhorn iteration (e.g., only one iteration).

% \begin{figure}
%     \centering
%     \includegraphics[width=0.8\linewidth]{Pictures/IPOT.PNG}
%     \caption{Caption}
%     \label{The dashed re}
% \end{figure}

\section{Relation with the proximal point method}
%Proposition

In this subsection, we shall demonstrate the equivalence and difference between the adaptive softassign and the proximal point method proposed by \cite{xie2020fast}. The linear convergence rate of the adaptive softassign methods can be inferred from the convergence of the proximal point method.  While the difference brings computational efficiency.

\begin{proposition}
\label{pro:relation}
The softassign transition \eqref{eq:transition} can solve
\begin{equation}
 {S}^{\beta_2}_X=\arg \max_{s \in \Sigma_{n\times n}}\langle X,  {S}\rangle-(\beta_2-\beta_1) D_h\left( {S},  {S}^{\beta_1}_X\right),
 \label{eq:transition_problem}
\end{equation}
\begin{equation}
    D_h(\mathbf{x}, \mathbf{y})=\sum_{i=1}^n x_i \log \frac{x_i}{y_i}-\sum_{i=1}^n x_i+\sum_{i=1}^n y_i.
\end{equation}
\end{proposition}
\begin{proof}
The solution of \eqref{eq:transition_problem} in the proximal point method is 
\begin{equation}
  \mathcal{P}_{sk}( {S}^{\beta_1}_X \circ \exp((\beta_2-\beta_1)X)). 
\end{equation}
According to the Hadamard-Equipped Sinkhorn Theorem, we have
\begin{equation}
    {S}^{\beta_1}_X = \exp(\beta_1 X) \circ SK^{(\exp(\beta_1 X))}.
\end{equation}
Then
\begin{equation}
    \begin{aligned}
    \mathcal{P}_{sk}&( {S}^{\beta_1}_X \circ \exp((\beta_2-\beta_1)X))\\
        &=\mathcal{P}_{sk}(\exp(\beta_1 X) \circ \exp((\beta_2-\beta_1)X))\\
        &= \mathcal{P}_{sk}(\exp(\beta_2 X))\\
        &={S}^{\beta_2}_X,
    \end{aligned}
\end{equation}
which is equivalent to the softassign transition \eqref{eq:transition}.

\end{proof}

According to Proposition \ref{pro:relation}, we can rewrite the iterative formula of the adaptive softassign as a proximal point method in \cite{xie2020fast}
\begin{equation}
 {S}^{(k)}_X=\arg \max_{s \in \Sigma_{n\times n}}\langle X,  {S}\rangle-(\Delta \beta) D_h\left( {S},  {S}^{(k-1)}_X\right),
 \label{eq:proximal}
\end{equation}
where $D_h(\cdot)$, the \textit{Bregman divergence}, is a regularization term to define the proximal operator. This indicates adaptive softassign is a variant of the proximal point method for problem \eqref{eq: entropic assignment} and enjoys a linear convergence rate \cite{xie2020fast}. 

Let us discuss the difference between adaptive softassign and the proximal point method. Adaptive softassign aims at obtaining a sub-optimal solution and $\beta_{\epsilon}$ with a given error bound, where $\beta_{\epsilon}$ can be used as a good initial $\beta_0$ in the next adaptive softassign in the whole graph matching process. While the proximal point method aims to find the exact solution, its efficiency is secondary and the change of $\beta$ is implicit. As to the computation aspect, the proximal point method solves \eqref{eq:transition_problem} according to
\begin{equation}
 {S}^{\beta_2}_X = \mathcal{P}_{sk}( {S}^{\beta_1}_X \circ \exp((\beta_2-\beta_1)X)). 
\end{equation}
Softassign transition only adapts a power operation and does not need the $X$, which indicates the change of $\beta$ more clearly. One can track and analyze the explicit change of $\beta$ easily.
\begin{figure*}[h]

  \centering
  \begin{subfigure}{0.45\linewidth}
  \includegraphics[width=1\linewidth]{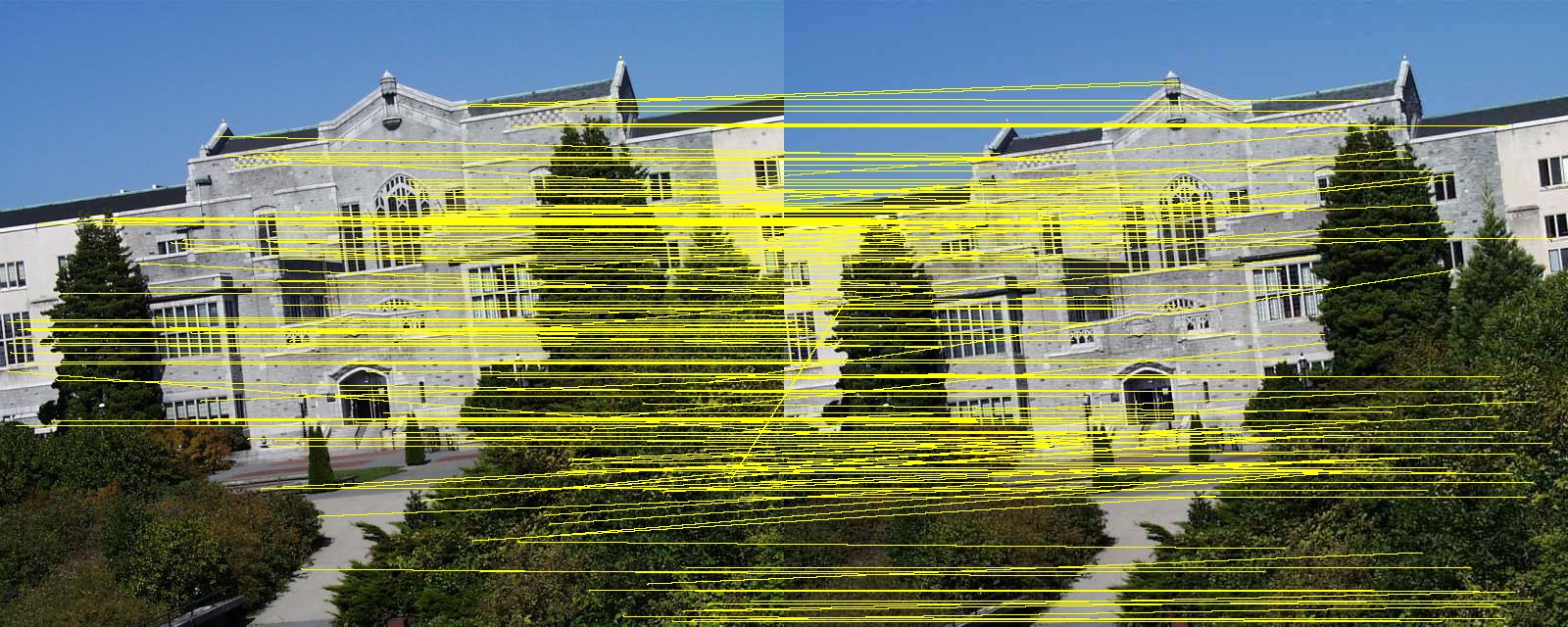}
    \caption{ubc}
  \end{subfigure}
  \hfill
  \begin{subfigure}{0.45\linewidth}
  \includegraphics[width=1\linewidth]{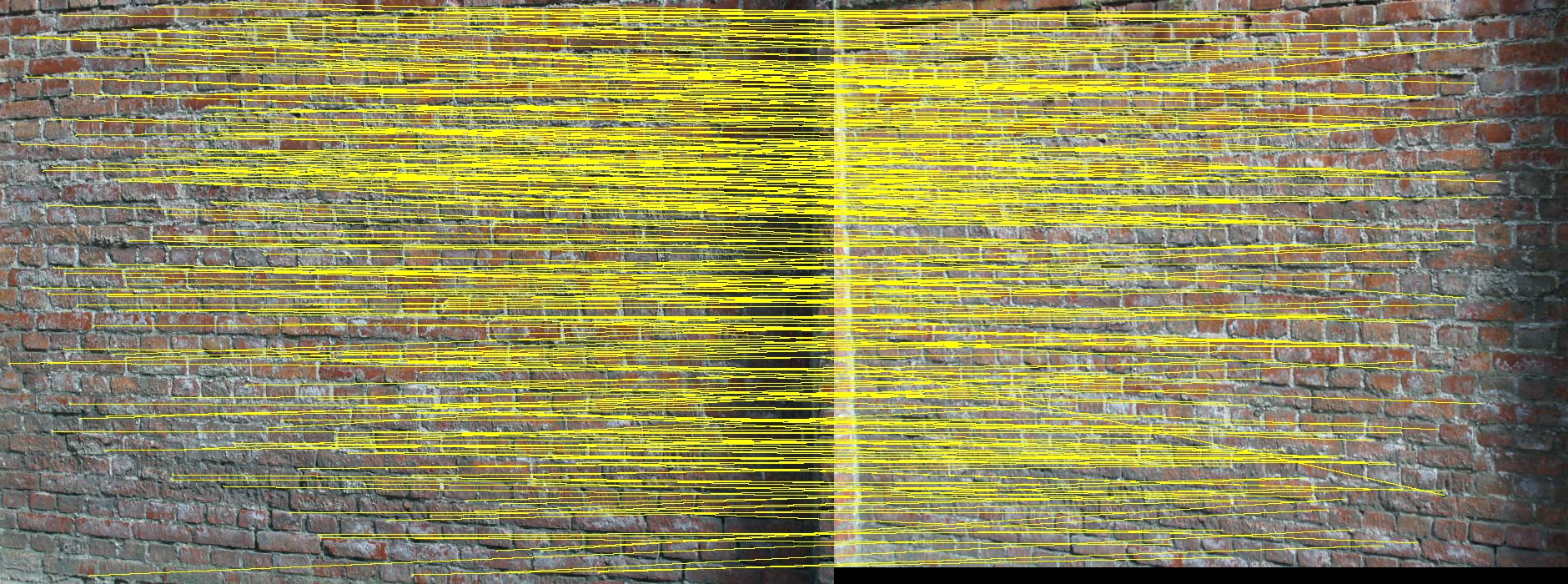}
    \caption{wall}
  \end{subfigure}
    \hfill
  \begin{subfigure}{0.45\linewidth}
  \includegraphics[width=1\linewidth]{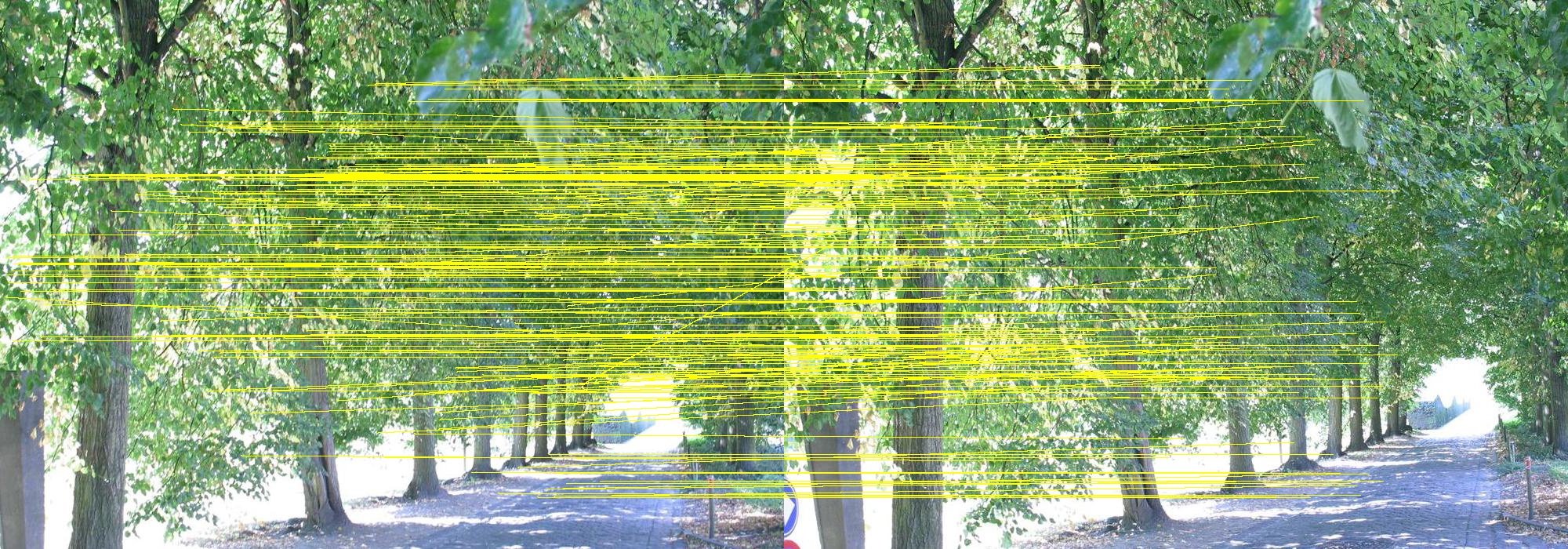}
    \caption{trees}
  \end{subfigure}
    \hfill
  \begin{subfigure}{0.45\linewidth}
  \includegraphics[width=1\linewidth]{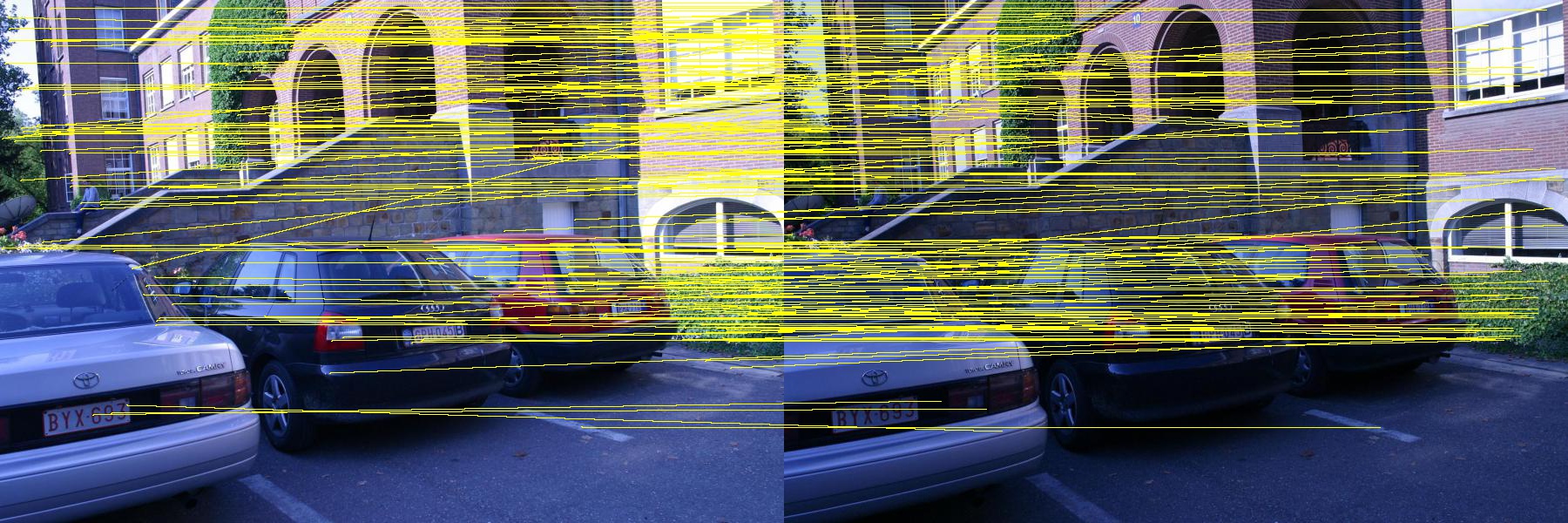}
    \caption{leuven}

  \end{subfigure}
  
  \caption{Graphs from real images matching. The yellow lines represent the correspondence between key points of the pictures.}
  \label{fig.real}
\end{figure*}

\section{Baselines and visualization of experiments}
Visualization of the matching results is shown in Figure \ref{fig.real}.
\textbf{Baselines:}
\begin{itemize}
    \item DSPFP \cite{lu2016fast} is a fast doubly stochastic projected fixed-point method with an alternating projection.
    \item GA \cite{gold1996graduated} can be considered a softassign-based projected fixed-point method with an outer annealing process.
    \item AIPFP \cite{leordeanu2009integer,lu2016fast} is an integer projected fixed point method with a fast greedy integer projection.
    \item SCG \cite{shen2022dyna} is a constrained gradient method with a dynamic softassign invariant to the nodes' number.
    \item GWL \cite{xu2019gromov} measures the distance between two graphs by Gromov-Wasserstein discrepancy and matches graphs by optimal transport.
    \item S-GWL \cite{xu2019scalable} is a scalable variant of GWL. It divides matching graphs into small graphs to match. (In Facebook network matching, we modify a parameter of S-GWL (beta: from 0.025 to 1) to avoid failure of partitioning. Otherwise, S-GWL becomes a very slow GWL).
    \item MAGNA++ \cite{vijayan2015magna++} is a global network alignment method for protein-protein interaction network matching, which focuses on node and edge conservation.
    \item GRASP \cite{hermanns2023grasp} aligns nodes based on functions derived from Laplacian matrix eigenvectors.
\end{itemize}

\clearpage
    \begin{center}
        \Huge \bfseries Updates 
    \end{center}

\textbf{Version 2.}
The reference [32], \textit{Dynamic softassign and adaptive parameter tuning for graph matching}, has been enhanced and renamed to \textit{CSGO: Constrained-Softassign Gradient Optimization for Large Graph Matching}.

\textbf{Version 3.}  
Correct the typo in the expression of the Sinkhorn projection by replacing
\begin{equation}
    \mathcal{P}_{\mathrm{sk}}(X) = X \circ (\mathbf{r}^T \otimes \mathbf{c})
\end{equation}
with the correct formulation
\begin{equation}
    \mathcal{P}_{\mathrm{sk}}(X) = X \circ (\mathbf{r} \otimes \mathbf{c}^T).
\end{equation}

\end{document}